\documentclass[11pt]{article}

\usepackage{layout}
\usepackage[makeroom]{cancel}
\usepackage{bbm}
\usepackage[affil-it]{authblk}
\usepackage{cite}
\usepackage{amsfonts}
\usepackage{amsmath}
\usepackage{amssymb}
\usepackage{amsthm}
\usepackage{graphicx} \usepackage{enumerate} \usepackage{multicol}
\usepackage{mathrsfs} \usepackage[all,cmtip]{xy}
\usepackage{enumerate}
\usepackage{cite}
\usepackage[dvipsnames]{xcolor}

\allowdisplaybreaks[4]

\newlength{\bibitemsep}
\newlength{\bibparskip}
\let\oldthebibliography\thebibliography
\renewcommand\thebibliography[1]{%
  \oldthebibliography{#1}%
  \setlength{\parskip}{\bibitemsep}%
  \setlength{\itemsep}{\bibparskip}%
}

\usepackage[colorlinks,linkcolor=blue,citecolor=blue,pagebackref,hypertexnames=false, breaklinks]{hyperref}

\newtheorem{theorem}{Theorem}[section]

\newtheorem{definition}[theorem]{Definition}

\newtheorem{corollary}[theorem]{Corollary}
\newtheorem{lemma}[theorem]{Lemma}
\newtheorem{remark}[theorem]{Remark}
\newtheorem{example}[theorem]{Example}
\newtheorem{examples}[theorem]{Examples}
\newtheorem{foo}[theorem]{Remarks}

\def\vint{\mathop{\mathchoice%
          {\setbox0\hbox{$\displaystyle\intop$}\kern 0.22\wd0%
           \vcenter{\hrule width 0.6\wd0}\kern -0.82\wd0}%
          {\setbox0\hbox{$\textstyle\intop$}\kern 0.2\wd0%
           \vcenter{\hrule width 0.6\wd0}\kern -0.8\wd0}%
          {\setbox0\hbox{$\scriptstyle\intop$}\kern 0.2\wd0%
           \vcenter{\hrule width 0.6\wd0}\kern -0.8\wd0}%
          {\setbox0\hbox{$\scriptscriptstyle\intop$}\kern 0.2\wd0%
           \vcenter{\hrule width 0.6\wd0}\kern -0.8\wd0}}%
          \mathopen{}\int}

\newcommand{\R}{\mathbb R}

\title{ BV spaces and the perimeters related to Schr\"{o}dinger operators with inverse-square potentials and applications to the rank-one theorem
      }

\vspace{0.5cm}

\author{   \  Yang Han, \ \  Jizheng Huang,
     \  \  Pengtao Li,\  \  Yu Liu\footnote{Corresponding author.}}
\vspace{-1.5cm}

\date{}

\begin{document}

\maketitle

{\bf Abstract:} For $a \ge - {( \frac{{d}}{2}- 1)^2} $ and $2\sigma=
{{d - 2}}-( {{{(d - 2)}^2} + 4a})^{1/2}$, let
$$\begin{cases}\mathcal{H}_{a}= - \Delta +
\frac{a} {{{{ | x  |}^2}}},\\
\mathcal{\widetilde{H}}_{\sigma}=
2\big( { - \Delta  + \frac{{{\sigma ^2}}} {{{{ | x
|}^2}}}}\big)\end{cases}$$ be two Schr\"odinger operators with
inverse-square potentials. In this
paper, on the domain $\Omega  \subset {\mathbb {R}^d}\backslash \{ 0\}, d\geq 2,$ %apart from the origin,
the ${\mathcal{H} _a}$-BV space  $\mathcal{B} {\mathcal{V}
_{{\mathcal{H} _a}}}(\Omega )$ and the
${\mathcal{\widetilde{H}}_{\sigma}}$-BV space $\mathcal{B}
{\mathcal{V} _{{\mathcal{\widetilde H} _\sigma}}}(\Omega )$ related
to $\mathcal{H}_{a}$ and $\mathcal{\widetilde{H}}_{\sigma}$ are
introduced, respectively. We investigate a series of basic
properties of $\mathcal{B} {\mathcal{V} _{{\mathcal{H} _a}}}(\Omega
)$ and $\mathcal{B} {\mathcal{V} _{{\mathcal{\widetilde H}
_{\sigma}}}}(\Omega )$. Furthermore, we prove that
${\mathcal{\widetilde{H}}_{\sigma}}$-restricted BV functions can be
characterized equivalently via their subgraphs. As applications, we
derive the rank-one theorem for
${\mathcal{\widetilde{H}}_{\sigma}}$-restricted BV functions.

\vspace{0.2cm}


\vspace{0.2cm}


\tableofcontents

\section{Introduction}
\hspace{0.6cm}
In this paper, we will discuss several basic questions of geometric
measure theory related to the Schr\"{o}dinger operator with
inverse-square potential:
\[\mathcal{H}_{a}   =  - \Delta  + \frac{a}{{{{ | x  |}^2}}},\quad
a \ge  - {\left( {{{\frac{d}{2} - 1}}} \right)^2}\] on the Euclidean
space $\mathbb{R}^d$ with $d \ge 2$. More precisely, we interpret
$\mathcal{H}_{a}$ as the Friedrichs extension of this operator
defined initially on $C_c^\infty (\mathbb{R}^d\backslash  \{ 0
 \})$ (cf. \cite{KMVZZ}). The operator $\mathcal{H}_{a}$ often
appears in mathematics and physics, and is usually used as scale
limit for more complex problems. The references \cite{BPST1, BPST2,
KSWW, VZ, ZZ} discuss several examples of this situation in physics.
These examples range from combustion theory to the Dirac equation with
Coulomb potentials, and the study of the disturbance of classical
space-time metric such as Schwarzschild and Reissner-Nordstr\"{o}m
metric. The appearance of $\mathcal{H}_{a}$ as a scale limit (from
both microscopic and astronomical aspects) is a signal of its unique
properties: $\mathcal{H}_{a}$ is scale-invariant. In particular, the
potential function and Laplace function are equally strong in every
length scale. Accordingly, problems involving $\mathcal{H}_{a}$
rarely obey simple perturbation theory. This is one of the reasons
why we (and many scholars before us) choose this particular operator
for further study.

The first aim of this paper is to investigate the class of functions
of bounded variation  related to $\mathcal{H}_{a}$. In the
literature, a function of bounded variation, simply a BV-function,
is a real-valued function whose total variation is finite.  In the
multi-variable setting, a function defined on an open subset
$\Omega\subseteq \mathbb R^{d}, d\geq 2$, is said to have bounded
variation provided that its distributional derivative is a
vector-valued finite Radon measure over the subset $\Omega$. Let
$\text{div}$ and $\nabla$ denote the divergence operator and the
gradient operator, respectively, where
\begin{equation}\label{eq1}\left\{\begin{aligned} \nabla u&:= \left( {\frac{{\partial u}}{{\partial {x_1}}},\frac{{\partial u}}{{\partial {x_2}}}, \ldots ,\frac{{\partial u}}{{\partial {x_d}}}} \right),\\
    \text{div}\varphi&:=\frac{\partial \varphi _1}{\partial{x_{1}}}+\frac{\partial
        \varphi _2}{\partial
        {x_{2}}}+\cdots+\frac{\partial \varphi _d}{\partial{x_{d}}}
\end{aligned}\right.\end{equation}
 for $u\in C_c^1(\Omega)$ and $ \varphi = (
{{\varphi _1},{\varphi _2}, \ldots ,{\varphi _d}} ) \in C_c^1(\Omega
,\mathbb{R}^{d} )$. A function $u\in L^{1}(\Omega)$ whose partial
derivatives in the sense of distributions are measures with finite
total variation $|Du|$ in $\Omega$ is called a function of bounded
variation, where
\begin{eqnarray*}
    |Du|(\Omega):=\sup\left\{\int_{\Omega}u\text{ div }\nu dx:\
    \nu=(\nu_{1},\ldots, \nu_{d})\in C^{\infty}_{c}(\Omega, \mathbb
    R^{d}), |\nu(x)|\leq 1, x\in\Omega\right\}<\infty.
\end{eqnarray*}
The class of all such functions will be denoted by $BV(\Omega)$. The
norm of $BV(\Omega)$ is defined as
$$\|u\|_{BV(\Omega)}:=\|u\|_{L^{1}(\Omega)}+|Du|(\Omega).$$
Importantly, the BV-functions form an algebra of possibly
discontinuous functions whose weak first partial derivatives exist
and are Radon measures-thanks to this nature, this algebra is
frequently used to define generalized solutions of nonlinear
problems involving functional analysis, ordinary and partial
differential equations, mathematical physics and engineering
science.

Essentially, the above definition of the BV-function  is
corresponding to the Laplace operator $\Delta$, where
$\mathcal{H}_{0}=-\Delta$, since $\Delta$ can be represented as
$\text{ div}(\nabla(\cdot))$. If $d=2$, at this time,  $a \ge -(
\frac{d}{2} - 1)^2$ implies that $a\ge 0$, then $\mathcal{H}_{a} $
is a Schr\"{o}dinger operator with the nonnegative potential and it
is positive semi-definite (e.g. cf. \cite[Section 3]{Auscher}). If
$d\ge 3$, the arguments of \cite[Section 1.1]{KMVZZ} imply that the
restriction $a \ge - {( {d}/{2}- 1)^2} $ ensures that the operator
$\mathcal{H}_{a}$ is positive semi-definite. Moreover,
$\mathcal{H}_{a} $ can be factorized as
\begin{equation*}
    \mathcal{H}_{a} = -\sum\limits_{i = 1}^d {({A_{ - i,a}}{A_{i,a}})} ,
\end{equation*}
where
$${A_{i,a}} = \frac{\partial }{{\partial{x_i}}} + \sigma
\frac{{{x_i}}}{{{{ | x  |}^2}}}, {A_{ - i,a}} =  \frac{\partial }
{{\partial{x_i}}} -\sigma \frac{{{x_i}}}{{{{ | x  |}^2}}}, 1 \le i
\le d $$ with
\[\sigma : = \frac{{d - 2}}{2} - \frac{1}{2}\sqrt {{{(d - 2)}^2} +
4a} .\] When $a=0$, ${A_{ - i,a}}$ and ${A_{i,a}}$ are exactly the
classical partial derivatives. This fact indicates that the
operators ${A_{i,a}}, 1 \le  | i  | \le d,$  play the same role as
the classical partial derivatives $ \frac{\partial
}{{\partial{x_i}}}$ in $\mathbb R^{d}$. Based on this observation,
we call ${A_{i,a}},1 \le  | i
 | \le d$, the generalized derivatives associated to
$\mathcal{H}_{a}$.

{ Throughout this paper,  unless otherwise specified, we always assume that
 $\Omega\subset \mathbb R^{d}\backslash \{ 0\}$ with $d\geq 2$ is a domain.}  For
$u\in C_c^1(\Omega)$ and $\varphi = ( {{\varphi _1},{\varphi _2},
\ldots ,{\varphi _d}} ) \in C_c^1(\Omega ,\mathbb{R}^{d} )$, we
introduce the following generalized gradient operator and
generalized divergence operator associated to $\mathcal{H}_{a}$:
$$ \left\{\begin{aligned}
    {\nabla _{{{{\mathcal H}}_a}}}u&:=({A_{1,a}}u, \ldots ,
    {A_{d,a}}u);\\
    \mathrm{div}_{\mathcal{H} _{a} } \varphi &:={A_{ - 1,a}}{\varphi _1} +
    \cdots  + {A_{ - d,a}}{\varphi _d},
\end{aligned}
\right. $$ which also gives
$$\mathcal{H}_{a}u=-\mathrm{div}_{\mathcal{H} _{a}}\big(\nabla _
{\mathcal{H}_a}u\big)=  - \Delta u + \frac{a}{{{{ | x
 |}^2}}}u\ \ \forall\  u\in C_c^2(\Omega).$$

Naturally, we introduce the class of bounded variation functions
related to ${\mathcal{H} _a}$ denoted by $\mathcal{B} {\mathcal{V}
_{{\mathcal{H}_a}}} (\Omega )$. In Section \ref{sec-2.1}, we
investigate some basic properties of $\mathcal{B} {\mathcal{V}
_{{\mathcal{H} _a}}}(\Omega )$ including the lower semicontinuity,
the structure theorem, the approximation via
$C^{\infty}_{c}$-functions, etc. In Theorem \ref{th2.5}, we prove
that $\mathcal{H}_{a}$-BV functions can be approximated by smooth
functions. It should be noted that in contrast with Theorem 2 in
\cite[Section 5.2.2]{EG}, we need to add a condition (\ref{eq4}) in
Theorem \ref{th2.5}, which can be obtained by the Hardy-Sobolev
inequality (cf. \cite {CKN84}). Moreover, the ${{\mathcal
H}_a}$-perimeter of $E\subseteq \Omega$ induced by $\mathcal{B}
{\mathcal{V} _{{\mathcal{H} _a}}}(\Omega )$ is introduced in Section
\ref{sec-2.2}, see (\ref{def-hp}) below. We obtain the following
coarea inequality for $\mathcal{H}_{a}$-BV function: if $u\in
{{\mathcal B}}{{{\mathcal V}}_{{{\mathcal H}}_a}}(\Omega )$, then
\begin{equation}\label{eq-1.1}
     | {{\nabla _{{{{\mathcal H}}_a}}}u}  |(\Omega ) \le
    \int_{ - \infty }^{ + \infty } {{P_{{{{\mathcal H}}_a}}}({E_t},
    \Omega )dt},
\end{equation}
where ${E_t}=\{x\in\Omega:\ u(x)>t\}$ for $t\in\mathbb{R}$, see
Theorem \ref{th-2.9}.

In Section \ref{sec-3}, our purpose is to establish  the converse of inequality
(\ref{eq-1.1}). Via choosing a sequence of approximation functions, we can deduce that
\begin{align*}
    \int_{ - \infty }^{ + \infty } {{P_{\mathcal{H}_{a}}}({E_t},
    \Omega )} dt&\le \sqrt 2 \int_\Omega  {\Big( { | {\nabla u}  (x)| + \frac{{
     | \sigma   |}}{{ | x  |}} | {u(x)}  |}\Big)dx}
   \end{align*} for $u\in {{\mathcal B}}{{{\mathcal V}}_{{\mathcal{\widetilde{H}}_{\sigma}}}}
     (\Omega )$.
However, it should be noted that $ | {\nabla u}  (x)| + \frac{|
\sigma|}{{ | x |}} | {u(x)} |$ can not be dominated by $| {{\nabla
_{{{{\mathcal H}}_a}}}u}(x) |$. In fact, for example, let
$u(x)=x^2_1,\sigma=-1$, then via computation, we have
$$ | {\nabla u}
(x)| + \frac{| \sigma|}{{ | x |}} | {u(x)}
|=2|x_1|+\frac{x^2_1}{|x|}$$
 and
$$| {{\nabla _{{{{\mathcal H}}_a}}}u}
(x)|=\Big(4x^2_1-\frac{3x^4_1}{|x|^2}\Big)^{{1}/{2}}.$$ Hence, the
previous fact holds true. Thus,  the converse of (\ref{eq-1.1}) does
not hold for ${{\mathcal B}}{{{\mathcal V}}_{{{\mathcal
H}}_a}}(\Omega )$. With this in mind, we introduce the following
subspace of ${{\mathcal B}}{{{\mathcal V}}_{{{\mathcal
H}}_{a}}}(\Omega )$.
 The
classical gradient operator $\nabla$  and the classical divergence
operator $\mathrm{div}$ in (\ref{eq1})   inspire  us to introduce
the following symmetric gradient operator and symmetric divergence
operator related with
$\mathcal{\widetilde{H}}_{\sigma}:=\mathrm{div}_{\mathcal{\widetilde{H}}_{\sigma}}(
{\nabla _{\mathcal{\widetilde{H}}_{\sigma}}}(\cdot))$, precisely,
$$\left\{\begin{aligned}
    {\nabla _{\mathcal{\widetilde{H}}_{\sigma}}} u&:= ({A_{1,a}}u, \ldots ,
    {A_{d,a}}u, {A_{ - 1,a}}u, \ldots ,{A_{ - d,a}}u),\\
    \mathrm{div}_{\mathcal{\widetilde{H}}_{\sigma}}\Phi&: ={A_{ - 1,a}}{\varphi _1} +
    \cdots  + {A_{ - d,a}}{\varphi _d} + {A_{ 1,a}}{\varphi _{d + 1}}
    +
    \cdots  + {A_{d,a}}{\varphi _{2d}}
\end{aligned}  \right. $$
for $u\in C_c^1(\Omega)$ and $\Phi  = ({\varphi _1},{\varphi _2},
\ldots ,{\varphi _{2d}}) \in C_c^1(\Omega ,{\mathbb{R} ^{2d}}).$ The
class of bounded variation functions related to
${\mathcal{\widetilde{H}}_{\sigma}}$ is denoted by $\mathcal{B}
{\mathcal{V}_{{\mathcal{\widetilde{H}}_{\sigma}}}}(\Omega )$. Based
on the definition of ${\nabla _{\mathcal{\widetilde{H}}_{\sigma}}},$
in Theorem \ref {coarea formula}, using the equivalence of $ |\nabla
\varphi|+\frac{{
     | \sigma   |}}{{ | x  |}} | {\varphi(x)}  |  $ and $|{\nabla
_{\mathcal{\widetilde{H}}_{\sigma}}} \varphi|$, we prove that
(\ref{eq-1.1}) can be improved to the following coarea formula:
\begin{equation}\label{eq-1.2}
     | {{
            \nabla _{{\mathcal{\widetilde{H}}_{\sigma}}}}u}  |(\Omega ) \approx \int_
    { - \infty }^{ + \infty } {{P_{{\mathcal{\widetilde{H}}_{\sigma} }}}({E_t},\Omega )
        dt} \quad \forall\  u \in {{\mathcal B}}{{{\mathcal V}}_{\mathcal{\widetilde{H}}_{\sigma}}}(\Omega ).
\end{equation}
By the aid of (\ref{eq-1.2}), we deduce that the Sobolev type inequality
\[{ \| f  \|_{{L^{{d}/{(d-1)}}}(\Omega )}} \lesssim  | {{\nabla
_{\mathcal{\widetilde{H}}_{\sigma}}}f}  |(\Omega )\  \;\forall \;f
\in \mathcal{BV}_{_{\mathcal{\widetilde{H}}_{\sigma}}}(\Omega )\] is
equivalent to the following isoperimetric inequality
\[  | E  |^{1-1/d}  \lesssim {{P_{\mathcal{\widetilde{H}}_{\sigma}}}(E,\Omega )}, \]
where $E$ is a bounded set with finite
${\mathcal{\widetilde{H}}_{\sigma}}$-perimeter in $\Omega$.

As an application, we further investigate the rank-one property for
$\mathcal{\widetilde{H}}_{\sigma} $-variations. The rank-one theorem
was first conjectured by L. Ambrosio and E. De Giorgi in \cite{GA}
(see also \cite{D,ACP}). This theorem is of great significance to
the application of vector variational problems (lower
semicontinuity, relaxation, approximation and integral
representation theorem, etc.) and partial differential equations. By
introducing new tools and using complex techniques in geometric
measure theory, G. Alberti first proved the rank-one theorem  in
\cite{A}. A simpler proof, based on the area formula and the
Reshetnyak continuity theorem, has been given in \cite{AA}.
Unfortunately, this proof works for particular BV functions only,
the monotone ones (gradients of locally bounded convex functions,
for instance). Two different proofs of rank-one theorem have been
found recently. One of them was proposed by G. De Philippis and F.
Rindler, who obtained a new proof from a profound PDEs result
\cite{DR}, and also proved a rank-one property for maps with bounded
deformation (BD) firstly. At the same time, A. Massaccesi and D.
Vittone in \cite{MV} provided another simpler proof of geometric
properties by virtue of the properties of subgraphs in Euclidean
space. Applying  properties related with the horizontal derivatives
of a real-valued function with bounded variation and its subgraph,
S. Don, A. Massaccesi and D. Vittone obtained a rank-one theorem for
the derivatives of vector-valued maps with bounded variation in a
class of Carnot groups $\mathbb G$ (see \cite{DMV}).

One of significant properties of BV functions in $\mathbb R^{d}$ is
that any BV function  can be characterized equivalently by its
subgraph  (cf. \cite{GMS}). In \cite{MV} and \cite{DMV}, such
subgraph  property is an important tool in the proof of the rank-one
theorem. In Section \ref{sec-4}, similar to the concept of
subgraph in $\mathbb R^{d}$, we introduce the definition of subgraph
of ${\mathcal{\widetilde{H}}_{\sigma}}$-BV function and  study the
subgraph properties of ${\mathcal{\widetilde{H}}_{\sigma}}$-BV
functions. We point out that in settings of Euclidean spaces
$\mathbb R^{d}$ and Carnot groups $\mathbb G$, for $\varphi\in
C^{1}_{c}(\Omega, \mathbb{R}^d)$, the integral of the divergence
$\text{div}\varphi$ is zero, i.e.,
\begin{equation}\label{eq-1.3}
    \int_{\Omega}\text{ div }\varphi(x)dx=0.
\end{equation}
However, due to the occurrence of the perturbation term
$\frac{\sigma{x_i}} {{{ { | x  |}^2}}}$ in
$\mathrm{div}_{\mathcal{\widetilde{H}}_{\sigma}}$, (\ref{eq-1.3})
does not hold for $\mathrm{div}_{\mathcal{\widetilde{H}}_{\sigma}}$.
We introduce the concept of
${\mathcal{\widetilde{H}}_{\sigma}}$-restricted variation and
perimeter to eliminate the influence of the potential part of the
operator ${{{\mathcal H}}_a}$ on the subgraph properties. Finally,
in Theorem \ref{subgraph-1}, we prove that a measurable function $u$
belongs to $\mathcal{B}
{\mathcal{V}^R_{{\mathcal{\widetilde{H}}_{\sigma}}}}(\Omega )$ if
and only if its subgraph is a set of finite ${{\mathcal
{\widetilde{H}}}_\sigma}$-restricted perimeter in ${\mathbb {R}^d}$.

In Section \ref{sec-5}, we devote ourselves to a closer analysis of the
distributional derivative of a
${\mathcal{\widetilde{H}}_{\sigma}}$-restricted BV function $u$. In
analogy with the results obtained in Euclidean space \cite{AFP} for
BV functions, we can write ${\nabla
_{\mathcal{\widetilde{H}}_{\sigma}}}u = \nabla
_{\mathcal{\widetilde{H}}_{\sigma}}^Au + \nabla _
{\mathcal{\widetilde{H}}_{\sigma}}^Su$, where $\nabla
_{\mathcal{\widetilde{H}}_{\sigma}}^Au$ is absolutely continuous
with respect to ${{{\mathcal L}}^d}$ and $\nabla
_{\mathcal{\widetilde{H}}_{\sigma}}^Su$ is singular with respect to
${{{\mathcal L}}^d}$. Based on the above decomposition of
derivatives, in Section \ref{sec-5}, the proof of rank-one theorem
related to $\mathcal{B} {\mathcal{V}^R
_{{\mathcal{\widetilde{H}}_{\sigma}}}} (\Omega )$ is given, see
Theorem \ref{ Rand-one }.

{\it Some notation}:
\begin{itemize}
    \item Throughout this article, we will use  $c$ and  $C$ to denote the
    positive constants, which are independent of main parameters and may
    be different at each occurrence.
    ${\mathsf U}\approx{\mathsf V}$ indicates that
    there is a constant $c>0$ such that $c^{-1}{\mathsf V}\le{\mathsf
        U}\le c{\mathsf V}$, whose right inequality is also written as
    ${\mathsf U}\lesssim{\mathsf V}$. Similarly, one writes ${\mathsf V}
    \gtrsim{\mathsf U}$ for ${\mathsf V}\ge c{\mathsf U}$.

    \item For convenience, the positive constant $C$
    may change from one line to another and this usually depends on the
    spatial dimension $d$, the indices $p$,
    and other fixed parameters.

    \item Let $ \mathbb{N}_0$ denote the non-negative integers. A $d$-dimensional
multi-index is a vector $\alpha\in \mathbb{N}^d_0$, meaning that,
$\alpha=(\alpha_1,\ldots,\alpha_d)$ for $\alpha_i\in \mathbb{N}_0$.
The derivative    of order $\alpha$ is defined by
$$\partial ^\alpha= \frac{\partial^{|\alpha| }}{\partial
x^{\alpha_1}_1\cdots\partial x^{\alpha_d}_d}$$
with $|\alpha|=\sum^d_{i=1}\alpha_i.$

    \item Let $\Omega\subset\mathbb{R}^{d}$ be an open set.  Throughout
    this article, we use ${C}(\Omega)$ to denote
    the space  of all continuous functions on $\Omega$.  Let $k
    \in\mathbb{N}\cup \{\infty\}$. The symbol  ${C}^{k}(\Omega)$
    denotes the class of all functions $f:\ \Omega \rightarrow
    \mathbb{R}$ with $k$ continuous partial derivatives. Denote by ${C}^
    {k}_{c}(\Omega)$  the class of all functions $f\in {C}^{k}
    (\Omega)$ with compact support.
\end{itemize}

\section{${{\mathcal H}_a}$-BV functions} \label{sec-2}

\subsection{Basic properties of $\mathcal{B} {\mathcal{V} _{{\mathcal
            {H} _a}}}(\Omega )$}\label{sec-2.1}
      \hspace{0.6cm}
In this section, we  introduce the ${{\mathcal
H}_a}$-BV space and investigate its properties.  The ${{\mathcal
H}_a}$-divergence of a vector-valued function
\[\varphi =({{\varphi _1},{\varphi _2}, \ldots ,{\varphi
_d}} ) \in C_c^1(\Omega ,{\mathbb R^d})\] is defined as
\[\mathrm{div}{_{\mathcal{H} _{a} }} \varphi ={A_{ - 1,a}}{\varphi
_1} +  \cdots  + {A_{ - d,a}}{\varphi _d}.\]
 By a simple computation, we have
\begin{eqnarray*}
{{\mathcal H}_a}:=-\mathrm{div}{_{{{{\mathcal H}}_a}}} (
{{\nabla_{{{{\mathcal H}}_a}}}u} ) &=&-\Big( {  \nabla -\sigma
\frac{x}{{{{ | x  |}^2}}}}\Big) \cdot \Big( {\nabla +\sigma
\frac{x}{{{{ |
x  |}^2}}}}\Big)u\\
& =& \Big( { - \Delta  - \sigma ( {d - 2 - \sigma } )\frac{1}{{{{ |
x  |}^2}}}} \Big)u \\
& =& \Big( { - \Delta +\frac{a}{{{{ | x  |}^2}}}} \Big)u,
\end{eqnarray*}
where
$$\sigma : = \frac{{d - 2}}{2} - \frac{1}{2}\sqrt {{{(d - 2)}^2} +
4a}.$$ Let $\Omega \subseteq \mathbb{R}^{d}$ be an open set. The
${{\mathcal H}_a}$-variation of $f \in {L^1}(\Omega )$ is defined by
\[ | {{\nabla _ {{{\mathcal H}_a}}}f}  |(\Omega ) :=
\mathop {\sup }\limits_{\varphi \in {\mathcal F(\Omega )}}  \left\{
{\int_\Omega  {f(x)\mathrm{div}{_{{{ \cal H}_a}}}\varphi (x)dx} }
 \right\},\] where ${\mathcal F}(\Omega )$ denotes the class of all
functions \[\varphi  = ( {{\varphi _1},{\varphi _2}, \ldots
,{\varphi _d}} ) \in C_c^1(\Omega ,{\mathbb R^d})\] satisfying
\[{ \| \varphi   \|_\infty } = \mathop {\sup }\limits_{x \in \Omega }
\Big\{( {{{ | {{\varphi _1}(x)}  |}^2} +  \cdots  + {{ | {{\varphi
_d}(x)}  |}^2}} )^{{1}/{2}}\Big\} \le 1.\]

An function $f \in {L^1}(\Omega )$ is said to have the ${{\mathcal
H}_a}$-bounded variation on $\Omega $ if
$$ | {{\nabla _{{{\mathcal H}_a}}}f}  |(\Omega ) < \infty ,$$
 and the collection of all such functions is
 denoted by $\mathcal{BV}_{{{\mathcal H}_a} }    ( \Omega    ) $.
 It follows from Remark \ref{remark} below that $\mathcal{BV}_{{{\mathcal H}_a} }(\Omega)$ is a Banach space with the norm \[ { \| f  \|_{{{\mathcal B}}
 {{{\mathcal V}}_{{{{\mathcal H}}_a}}}(\Omega )}} = { \| f  \|_{
 {L^{1}}(\Omega )}}+  | {{\nabla _{{{{\mathcal H}}_a}}}f}  |
 (\Omega ) . \]

For the sake of research, we give the definition of the   Sobolev
space associated with ${{\mathcal H}_a}$. \cite{KMVZZ} has also
studied the   Sobolev space defined in terms of the operator
 $({\mathcal H}_a)^{{s}/{2}}$ for  $0<s<1$.

\begin{definition}\label{def-Sobolev}
Suppose $\Omega$ is an open set in $\mathbb{R} ^{d}$ for $d \ge 2$.
Let $1 \le p \le \infty$. The Sobolev space $W_{{{\mathcal
H}_a}}^{k,p}(\Omega )$ associated with ${{\mathcal H}_a}$ is defined
as the set of all functions $f \in {L^p}(\Omega )$ such that
\[{A_{{j_1},a}} \cdots {A_{{j_m},a}}f \in {L^p}(\Omega),\quad 1
\le {j_1}, \ldots ,{j_m} \le d\ \&\ 1 \le m \le k.\] The norm of $f \in
W_{{{\mathcal H}_a}}^{k,p}(\Omega )$ is given by \[{ \| f
 \|_{W_{{{\cal H}_a}}^{k,p}}}: = \sum\limits_ {1 \le {j_1},
\ldots ,{j_m} \le d,1 \le m \le k} {{{ \| {{A_{{j_1},a}} \cdots
{A_{{j_m},a}}f}  \|}_{{L^p}(\Omega)}}}  + { \| f
\|_{{L^p}(\Omega)}},\] where $a \ge  - {( {{{\frac{d}{2}- 1}}}
)^2}$.
\end{definition}
 In what follows,  we will collect some properties of the space
$\mathcal{BV}_ {{\mathcal H}_a}(\Omega)$.

\begin{lemma}\label{semicontinuity}
   \item{{\rm (i)}} Suppose that $f \in W_{{{\mathcal H}_a}}^{1,1}
   (\Omega )$. Then
\begin{equation*}
 | {{\nabla _{{{{\mathcal H}}_a}}}f}  |(\Omega ) = \int_\Omega
{ |    {{\nabla _{{{{\mathcal H}}_a}}}f(x)}  |dx} .
\end{equation*}
  \item{{\rm (ii)}} (Lower semicontinuity) Suppose that ${f_k} \in {{\mathcal B}}
{{{\cal V}}_{{{{\mathcal H}}_a}}}(\Omega ),k \in \mathbb N$ and
${f_k} \to f$ in $ L_{loc}^1(\Omega ) $. Then \[ | {{\nabla
_{{{{\mathcal H}}_a}}}f}  |(\Omega ) \le \mathop {\lim \inf
}\limits_{k \to \infty }  | {{\nabla _{{{{\mathcal H} }_a}}}{f_k}}
|(\Omega ). \]
\end{lemma}

\begin{proof}
   (i) For every $\varphi\in C_c^{1}(\Omega,\mathbb R^{d})$ with
${ \| \varphi   \|_{{L^\infty }(\Omega )}} \le 1$, we have
\begin{align*}
     \left|\int_{\Omega}  f(x) \mathrm{div}_{\mathcal H_a} \varphi(x)  dx \right|
    & =  \left|\int_{\Omega} \nabla_{\mathcal H_a} f(x)\cdot \varphi(x)
    dx \right|\leq \int_{\Omega}|\nabla_{\mathcal H_a} f(x)| dx.
\end{align*}
By taking the supremum over $\varphi$, it is obvious that
$$ |
{{\nabla _{{{{\mathcal H}}_a}}}f}  |(\Omega ) \le \int_\Omega { |
{{\nabla _{{{{\mathcal H}}_a}}}f} (x) |dx}.$$

Define $\varphi\in L^\infty(\Omega,\mathbb R^{d})$ as follows:
\begin{equation} \label{eqq4}\varphi(x):=
    \left\{\begin{aligned}
        \frac{\nabla_{\mathcal H_a} f(x)}{|\nabla_{\mathcal H_a} f(x)|},\ &
        \text{if $x\in\Omega$ and $\nabla_{\mathcal H_a} f(x)\neq 0$,} \\
        0,\ & \text{otherwise.}
\end{aligned}\right.\end{equation}
It is easy to see  that ${ \| \varphi   \|_{{L^\infty }(\Omega )}}
\le 1$. Moreover,  we can obtain the approximating smooth fields
${\varphi _n} : = ( {{\varphi _{n,1}}, \ldots ,{\varphi _{n,d}}} )$
such that $\varphi_n\to \varphi$ pointwise  as $n\to\infty$, with
$\|\varphi_n\|_{L^\infty(\Omega)} \leq 1$ for all $n \in \mathbb N$.
In fact, for $m>0$, let
$$\phi_{m,i}(x):=\varphi_{i}(x)
\chi_{B(0,m)\cap\Omega}(x),\, x\in\Omega,$$ where $i=1,\ldots,d$ and
$\chi_{B(0,m)\cap\Omega}(x)$ is the characteristic function of
$B(0,m)\cap\Omega$. Then $\phi_{m,i}(x)\rightarrow\varphi_{i}(x)$ as
$m\rightarrow \infty$ for any  $x\in\Omega$. For any $\epsilon>0$,
there exists a sufficiently large $m_{0}$ such that
$$|\phi_{m_{0},i}(x)-\varphi_{i}(x)|<\epsilon/2.$$
 Choose
$\{\psi_{n}\}_{n>0}\subset C^{\infty}_{c}(\Omega)$ as an identity approximation and define
$$\varphi_{n,i}(x):=\psi_{n}\ast\phi_{m_{0},i}(x),\quad x\in\Omega.$$ Then
$\{\varphi_{n,i}\}_{n\in  \mathbb{N}} \subset C^\infty_c(\Omega)$
such that
$\lim\limits_{n\rightarrow\infty}\varphi_{n,i}(x)=\phi_{m_{0},i}(x)$
for any  $x\in\Omega$. This indicates that for the above
$\epsilon>0$, there exists $N>0$ such that for $n>N$,
$$|\varphi_{n,i}(x)-\phi_{m_{0},i}(x)|<\epsilon/2$$ and
$$|\varphi_{n,i}(x)-\varphi_{i}(x)|\leq |\phi_{m_{0},i}(x)-\varphi_{i}(x)|+|\varphi_{n,i}(x)-\phi_{m_{0},i}(x)|<\epsilon,$$
i.e., $\varphi_{n,i}\rightarrow \varphi_{i}$ pointwise as
$n\rightarrow \infty$, where $i=1,\ldots,d$. Then the desired
approximating smooth fields $\varphi_n$ can be obtained. Also, it
follows from the facts $\psi_{n}\in C^{\infty}_{c}(\Omega)$ and
$\|\varphi\|_{L^\infty(\Omega)}\leq 1$ that
$$|\varphi_{n}(x)|= \int_{\Omega}|\psi_{n}(y)|\cdot|\phi_{m,i}(x-y)|dy\le 1,$$
i.e., $\|\varphi_n\|_{L^\infty(\Omega)} \leq 1$. Combining the definition of $  | {{\nabla _{{{{\mathcal H}}_a}}}f}
|(\Omega )$ with integration by parts derives that for every $n\geq
1$,
\begin{align*}
     | {{\nabla _{{{{\mathcal H}}_a}}}}f  |(\Omega )
    &\ge \int_\Omega  {f(x)\mathrm{div}{_{{{{\mathcal H}}_a}}}{\varphi _n}(x)dx} \\
    &= \int_\Omega  {f(x) \left\{ {\Big(   \frac{\partial }{{\partial{x_1}}} - \sigma \frac{{{x_1}}}{{{{
        | x  |}^2}}}\Big){\varphi _{n,1}}(x) +  \cdots  + \Big(   \frac{\partial }{{\partial{x_d}}}
       -\sigma \frac{{{x_d}}}{{{{ | x  |}^2}}}\Big){\varphi _{n,d}}(x)}
        \right\}dx} \\
    &= -\int_\Omega  {{\nabla _{{{{\mathcal H}}_a}}}f(x)\cdot{\varphi _n}(x)dx}.
\end{align*}
Using the dominated convergence theorem and the definition of
$\varphi$ in (\ref{eqq4}), we have
\[ | {{\nabla _{{{{\mathcal H}}_a}}}f}  |(\Omega ) \ge \int_\Omega
{ | {{\nabla _{{{{\mathcal H}}_a}}}f} (x) |dx}  \] by letting $n\to
\infty$.

(ii) Fix $\varphi\in C^1_c(\Omega,\mathbb R^{d})$ with ${ \| \varphi
\|_{{L^\infty }(\Omega )}} \le 1$. By the definition  of $|
\nabla_{\mathcal{H} _a} f_k |(\Omega)$, we have
$$ | {{\nabla _{{{{\mathcal H}}_a}}}{f_k}}  |(\Omega ) \ge \int_
\Omega  {{f_k}(x)\mathrm{div}{_{{{{\mathcal H}}_a}}}\varphi (x)dx}.$$
Since $\varphi\in C^1_c(\Omega,\mathbb R^{d})$ and $0\notin \Omega$, there exists a constant $c>0$ such that $|x|>c$ for $x\in \text{supp }\varphi$, which gives
$$|\mathrm{div}{_{{{{\mathcal H}}_a}}}\varphi(x)|\leq \sum^{d}_{i=1}\Big\{|\frac{\partial\varphi}{\partial x_{i}}(x)|+|\sigma|\frac{|x_{i}|}{|x|^{\alpha}}|\varphi(x)|\Big\}\leq C.$$
This indicates that
\begin{eqnarray*}
\left|\int_ \Omega {{f_k}(x)\mathrm{div}{_{{{{\mathcal
H}}_a}}}\varphi (x)dx}-\int_\Omega {{f}(x)\mathrm{div}{_{{{{\mathcal
H}}_a}}} \varphi (x)dx} \right|\leq
C\int_{\Omega}|f_{k}(x)-f({x})|dx,
\end{eqnarray*}
which, together with the convergence of ${ \{ {{f_k}}  \}_{k \in \mathbb{N} }}$ to $f$ in
$L_{loc}^1(\Omega ) $, implies that
\[\mathop {\lim \inf } \limits_{k \to \infty }  | {{\nabla
_{{{{\mathcal H}}_a}}}{f_k}}   |(\Omega ) \ge \lim_{k \to \infty}
\int_ \Omega {{f_k}(x)\mathrm{div}{_{{{{\mathcal H}}_a}}}\varphi
(x)dx}=\int_\Omega {{f}(x)\mathrm{div}{_{{{{\mathcal H}}_a}}}
\varphi (x)dx} .\] Therefore, (ii) can be proved by the definition
of $  | {{\nabla _ {{{{\mathcal H}}_a}}}f}  |(\Omega )$ and the
arbitrariness of such functions $\varphi$.
\end{proof}

\begin{remark}\label{remark}
The space $(\mathcal{BV}_{\mathcal{H}_a}(\Omega),{ \|  \cdot
  \|_{\mathcal B{\mathcal V_{{\mathcal H_a}}}(\Omega )}})$ is a
Banach space. Firstly, it is easy to check  that ${ \|  \cdot
 \|_{\mathcal B{\mathcal V_ {{\mathcal H_a}}}(\Omega )}}$ is a
norm. Secondly,  let $ \{ {{f_n}}  \}_ {n\in\mathbb N}\subset
\mathcal{BV}_{\mathcal H_a}(\Omega)$ be a Cauchy sequence. Then $ \{
{{f_n}}  \}_{n\in\mathbb N}$ is a Cauchy sequence in the Banach
space $L^1(\mathbb R^d)$. Finally, using the lower semicontinuity of
${{\mathcal H}_a}$-BV functions (cf.  Lemma \ref{semicontinuity}),
we obtain the desired proof.
\end{remark}

The following lemma gives the structure theorem for ${{\mathcal
H}_a}$-BV functions and it can be proved by the Hahn-Banach theorem
and the Riesz representation theorem in \cite {AFP}.
\begin{lemma}\label{Structure}
     Let $u \in {{\mathcal B}}
    {{{\mathcal V}}_{{{{\mathcal H}}_a}}}(\Omega )$. There
    exists a unique $\mathbb R^{d}$-valued finite Radon measure ${\mu _
    {{\mathcal{H} _{a}}}}$ such that \[\int_\Omega  u (x)\mathrm{div}
    {_{{\mathcal{H}_{a}}}}\varphi (x)dx = \int_\Omega  \varphi  (x)
    \cdot d{\mu _{{\mathcal{H} {a}}}}(x)\] for any $\varphi\in C^{
    \infty}_c(\Omega, \mathbb R^{d})$ and \[ | {\nabla _{{
    \mathcal{H} _a}}}u | (\Omega ) = |{\mu _{{\mathcal{H} _{a}}}
    }|(\Omega ),\]
    where $|{\mu _{{{{\mathcal H}}_a}}}|$ is the total variation of
    the measure ${\mu _{{{{\mathcal H}}_a}}}$.
\end{lemma}

Combining the mean value theorem of multivariate functions with an
additional condition (\ref{eq4}), we can obtain the following
approximation result for the ${{\mathcal H}_a}$-variation. It should
be noted that the condition (\ref{eq4}) is added due to the
singularity of the perturbation term $\sigma\frac{x_i} {{{ { | x
 |}^2}}}$ in $\mathrm{div}_{\mathcal{ {H}}_{a}}$ and
we explain the relation between  the condition (\ref{eq4}) and the
Hardy-Sobolev inequality.

\begin{theorem}\label{th2.5}
    Let $\Omega\subset{\mathbb{R} ^d}$ be an open and bounded domain.
    Assume that $u \in {{\mathcal B}}
    {{{\mathcal V}}_{{{{\mathcal H}}_a}}}(\Omega )$ with
    \begin{equation}
        \label{eq4} {\int_\Omega  { | {u(y)}  | {{ | y
      |}^{ - 2}}dy}}<\infty,
    \end{equation}
    then there exists a sequence ${ \{ {{u_h}}  \}_
    {h \in \mathbb{N} }} \in {{\mathcal B}}{{{\mathcal V}}_{{{{\mathcal H}}_a}}}
    (\Omega ) \cap C_c^\infty (\Omega ) $ such that $$\mathop {\lim }\limits_{h
    \to \infty }{ \| {{u_h} - u}  \|_{{L^1}(\Omega)}} = 0$$ and
    \[\mathop {\lim }\limits_{h \to \infty } \int_\Omega  { | {{\nabla _{{{
    {\mathcal H}}_a}}}{u_h}(x)}  |dx}  =  | {{\nabla _{{{{\mathcal H}}
    _a}}}u} |(\Omega ).\]
\end{theorem}

\begin{proof}
The following proof   is similar to that of
\cite[Section 5.2.2, Theorem 2]{EG}, but different from  its proof  we need to
use the mean value theorem of multivariate functions and the
condition (\ref{eq4}).

Via the lower semicontinuity of  ${{\mathcal H}_a}$-BV functions, it
suffices to show that for $\varepsilon
> 0$, there exists a function ${u_\varepsilon } \in {C^\infty }(\Omega )$
such that
\begin{equation*}
\int_\Omega  { | {{u_\varepsilon }(x) - u(x)}  |dx}  < \varepsilon
\end{equation*}
and
\begin{equation*}
 | {{\nabla _{{{{\cal H}}_a}}}{u_\varepsilon }}  |(\Omega ) \le
 | {{\nabla _{{{{\cal H}}_a}}}u}  |(\Omega ) + \varepsilon.
\end{equation*}

Fix $\varepsilon  > 0$. Given a positive integer $m$, let \[{\Omega
_j}: =  \Big\{ { {x \in \Omega }:\ \mathrm{dist}(x,\partial \Omega ) >
\frac{1}{{m + j}}}  \Big\} \cap B(0,m + j),\ \ j\in \mathbb{N},
\] where $\mathrm{dist}(x,\partial \Omega )=\inf\{|x-y|:\ y\in \partial \Omega \}$. In fact, let ${ \{ {{\Omega _j}}  \}_{j \in \mathbb{N} }}$
be a sequence of subdomains of $\Omega $ such that ${\Omega _j} \subset {\Omega
_{j + 1}} \subset \Omega $, $j \in \mathbb{N} $ and $\mathop  \cup
\limits_ {j = 0}^\infty  {\Omega _j} = \Omega$. Since $ | {{\nabla
_{{{{\cal H}}_ a}}}u}  |(\cdot ) $ is a measure, we can choose a
$m\in \mathbb{N}$ so large that
\begin{equation}\label{eq-2.2}
 | {{\nabla _{{{{\cal H}}_a}}}u}  |(\Omega \backslash { \Omega
_0}) < \varepsilon .
\end{equation}

Let ${U_0}: = {\Omega _0}$ and ${U_j}: = {\Omega _{j + 1}}\backslash
{{\bar \Omega }_{j - 1}} $ for $j \ge 1$. Following the proof
of \cite[Section 5.2.2, Theorem 2]{EG}, we conclude that there is a
partition of unity subordinate to the covering ${ \{ {{U_j}} \}_{j
\in \mathbb N}} $. Thus, there exist functions ${ \{ {{f_j}} \}_{j
\in \mathbb N}} \in C_c^\infty ({U_j})$ such that $0 \le {f_j} \le
1$, $j \ge 0$ and $\sum\limits_{j = 0}^\infty  {{f_j}} = 1$ on
$\Omega$.

Given $\varepsilon  > 0$ and $u \in {L^1}(\Omega)$,
extended to zero out of $\Omega $, we define the usual
regularization \[{u_\varepsilon}(x): = \frac{1}{{{\varepsilon ^d}}}
\int_{B(x,\varepsilon )} {\eta (\frac{{x - y}}{\varepsilon
})u(y)dy},\] where $\eta  \in C_c^\infty ({\mathbb{R}^d})$ is a
nonnegative radial  function satisfying $\int_{{\mathbb{R} ^d}}{\eta
(x)dx}  = 1$ and ${\rm{supp}}\eta \subset B(0,1)$. Then for each
$j$, there exists $0 < {\varepsilon _j} < \varepsilon $, such that
\begin{equation}\label{eq-2.1}
     \left\{
    \begin{aligned}
    \hspace{-0.5cm}
       {{\rm{supp}}( {{{( {{f_j}u} )}_{{\varepsilon _j}}}}
        ) \subseteq {U_j},}\\
       {\int_\Omega  { | {{{({f_j}u)}_{{\varepsilon _j}}}(x) - {f_j}u}(x)
        |dx}  < \varepsilon {2^{ - (j + 1)}},}\\
       {\;\int_\Omega  { | {{{(u\nabla {f_j})}_{{\varepsilon _j}}}(x)  -
        u\nabla {f_j}} (x)  |dx}  < \varepsilon {2^{ - (j + 1)}}.}
    \end{aligned}
\right.\end{equation}

Define \[{v_\varepsilon }(x): = \sum\limits_{j = 0}^\infty
{{{(u{f_j})}_ {{\varepsilon _j}}}}(x).\] Clearly, ${v_\varepsilon }
\in {C^\infty }(\Omega )$ and $u = \sum\limits_{j = 0}^\infty
{u{f_j}}$. Therefore, by a simple computation, we can get
\[{ \| {{v_\varepsilon } - u}  \|_ {{L^1}(\Omega )}} \le
\sum\limits_{j = 0}^\infty  {\int_\Omega  {|{{( {{f_j}u}
)}_{{\varepsilon _j}}}(x) - {f_j}(x)u(x)|dx < \varepsilon } }. \]
Consequently, $ {v_\varepsilon } \to u$ in
${L^1}(\Omega )$ as $\varepsilon  \to 0$.  Now, let $\varphi  \in C_c^1(\Omega, {\mathbb{R}
^d})$ satisfying $ | \varphi   | \le 1$. Then
\begin{align*}
    &\int_\Omega  {{v_\varepsilon }(x)\mathrm{div}{_{{{{\mathcal H}}_a}
    }}\varphi (x)dx}\\
    &= \int_\Omega  {\Big(\sum\limits_{j = 0}^\infty  {{{(u{f_j})}_
    {{\varepsilon _j}}}}(x) \Big)\mathrm{div}{_{{{{\cal H}}_a}}}\varphi (x)dx} \\
    &= \sum\limits_{j = 0}^\infty  {\int_\Omega  {{{(u{f_j})}_{{
    \varepsilon _j}}}(x)\Big({A_{ - 1,a}}{\varphi _1}(x) + {A_{ - 2,a}}{\varphi _2}(x) +
    \cdots  + {A_{ - d,a}}{\varphi _d}(x)\Big)dx} } \\
    &: = I + II,
\end{align*}
where
$$\left\{\begin{aligned}
I&:= \sum\limits_{j = 0}^\infty  {\int_\Omega
{{{(u{f_j})}_{{\varepsilon _
j}}}(x)\Big(\frac{\partial}{\partial{x_1}}{\varphi _1}(x) +
\frac{\partial}{\partial x_2}{\varphi _2}(x) +  \cdots  +
\frac{\partial}{\partial x_d}{\varphi _d}(x)\Big)dx} };\\
II&:=-\sum\limits_{j = 0}^\infty  {\int_\Omega
{{{(u{f_j})}_{{\varepsilon _
    j}}}(x)\Big(\frac{{\sigma {x_1}}}{{{{ | x  |}^2}}}{\varphi _1}(x) +
    \frac{{\sigma {x_2}}}{{{{ | x  |}^2}}}{\varphi _2} (x)+  \cdots  +
    \frac{{\sigma {x_d}}}{{{{ | x  |}^2}}}{\varphi _d}(x)\Big)dx} }.
\end{aligned}\right.$$

As for $I$, we   get
\begin{align*}
    I &=   \sum\limits_{j = 0}^\infty  {\int_\Omega  {{{(u{f_j})}_{{\varepsilon _j
    }}}(x)\mathrm{div}\varphi(x) dx} } \\
    &=   \sum\limits_{j = 0}^\infty  {\int_\Omega  {(u{f_j})(y)\mathrm{div}({\eta _
    {{\varepsilon _j}}} * \varphi (y))dy} } \\
    &=    {\sum\limits_{j = 0}^\infty  {\int_\Omega  {u(y)\mathrm{div}({f_j}({\eta _
    {{\varepsilon _j}}} * \varphi ))(y)dy - } } \sum\limits_{j = 0}^\infty  {\int_
    \Omega  {u(y)\nabla ({f_j}({\eta _{{\varepsilon _j}}} * \varphi ))(y)dy} } }
      \\
    &=    {\sum\limits_{j = 0}^\infty  {\int_\Omega  {u(y){\rm{div}}({f_j}({\eta _
    {{\varepsilon _j}}}*\varphi ))(y)dy - } } \sum\limits_{j = 0}^\infty  {\int_\Omega
    {\varphi(y)  \Big( {{\eta _{{\varepsilon _j}}}*(u\nabla {f_j})(y) - u\nabla {f_j}}
     (y)\Big)dy} } }   \\
    &: = {I_1} + {I_2},
\end{align*}
where in the last equality we have used the fact that $\sum\limits_{j = 0}^\infty  {\nabla
{f_j}}(x)  = 0$ on $\Omega $. Actually, when ${
\| \varphi    \|_{{L^\infty }}} \le 1$, it holds that $ |
{{f_j}({\eta _{{\varepsilon _j}}}
* \varphi )} (x) | \le 1,\ j \in \mathbb{N} ,$ and each point in $\Omega $ belongs
to at most three of the sets $ \{ {{U_j}}  \}_0^\infty $. Moreover,
it follows from (\ref {eq-2.1}) that $ | {{I_2}}
 | < \varepsilon $.

As for $II$, we  change the order of integration to get
\begin{eqnarray*}
II &=&-\sum\limits_{j = 0}^{\infty}\int_{\Omega}
{{{(u{f_j})}_{{\varepsilon _ j}}}(x)\left\{\frac{{\sigma {x_1}}}{{{{
| x  |}^2}}}{\varphi _1}(x) + \frac{{\sigma {x_2}}}{{{{ | x
|}^2}}}{\varphi _2}(x) +  \cdots +
\frac{{\sigma {x_d}}}{{{{ | x  |}^2}}}{\varphi _d}(x)\right\}dx}  \\
&=& -\sum\limits_{j = 0}^\infty  {\int_\Omega  {\int_\Omega
{\frac{1}{{\varepsilon _j^d}}\eta ( {\frac{{x - y}}{{{\varepsilon
_j}}}})u(y){f_j}(y) \left( {\sum\limits_{k = 1}^d {\frac{{\sigma
{y_k}}}
{{{{ | y  |}^2}}}{\varphi _k}(x)} }  \right)dy} } dx} \\
    &&-\sum\limits_{j = 0}^{\infty}\int_{\Omega}\int_\Omega  {\frac{1}
    {{{\varepsilon _j^d}}}\eta ( {\frac{{x - y}}{{{\varepsilon _j}}}}
    )u(y){f_j}(y) \left\{ {\sum\limits_{k = 1}^d {\left( {\frac{{
    \sigma {x_k}}}{{{{ | x  |}^2}}} - \frac{{\sigma {y_k}}}{{{{
     | y  |}^2}}}} \right){\varphi _k}(x)} }  \right\}dy} dx  \\
    &=& -\sum\limits_{j = 0}^\infty  {\int_\Omega  {u(y){f_j}(y) \left( \sum
    \limits_{k = 1}^{d} \frac{\sigma {y_k}}{{{{ | y  |}^2}}}{{\varphi _k}\ast {\eta _{{\varepsilon _j}}}(y)}   \right)dy}
    } \\
    &&-\sum\limits_{j = 0}^\infty  {\int_\Omega  {\int_\Omega  {\frac{1}{{
    {\varepsilon _j^d}}}\eta ( {\frac{{x - y}}{{{\varepsilon _j}}}}
    )u(y){f_j}(y) \left\{ {\sum\limits_{k = 1}^d {\left( {\frac{{\sigma
    {x_k}}}{{{{ | x  |}^2}}} - \frac{{\sigma {y_k}}}{{{{ | y
     |}^2}}}} \right){\varphi _k}(x)} }  \right\}dy} dx} }.
\end{eqnarray*}

Consequently, the above estimate of the term ${I_2}$ shows that
\[ \left| {\int_\Omega  {{v_\varepsilon }(x)\mathrm{div}{_{{{{\mathcal H}}_a
}}}\varphi (x)dx} }  \right|{\rm{ = }} | {{I_1} + {I_2} + II}| \le {J_1}
+ {J_2} + \varepsilon , \]
 where
\[{J_1} :=  \left| {  \sum\limits_{j = 0}^\infty  {\int_\Omega  {u(y)\mathrm{div}({f_j}
({\eta _{{\varepsilon _j}}} * \varphi ))dy} }  - \sum\limits_{j =
0}^\infty {\int_\Omega  {u(y){f_j}(y) \left( {\sum\limits_{k = 1}^d
{\frac{{\sigma {y_k}}}{{{{ | y  |}^2}}}( {{\varphi _k} * {\eta _{{
\varepsilon _j}}}(y)} )} }  \right)dy} } }  \right|\]
 and
\[{J_2} :=  \left|- {\sum\limits_{j = 0}^\infty  {\int_\Omega  {\int_\Omega
{\frac{1}{{{\varepsilon _j^d}}}\eta ( {\frac{{x - y}}{{{\varepsilon
_j}}}} )u(y){f_j}(y) \left( {\sum\limits_{k = 1}^d {( {\frac{{\sigma
{x_k} }}{{{{ | x  |}^2}}} - \frac{{\sigma {y_k}}}{{{{ | y
 |}^2}} }} ){\varphi _k}(x)} }  \right)dy} dx} } }  \right|. \]

Moreover,
\begin{align*}
    {J_1}
    &=  \left| {   \sum\limits_{j = 0}^\infty  {\int_\Omega  {u(y)\mathrm{div}({f_j}
    ({\eta _{{\varepsilon _j}}} * \varphi ))dy} }  -\sum\limits_{j = 0}^
    \infty  {\int_\Omega  {u(y){f_j}(y) \left( {\sum\limits_{k = 1}^d {\frac
    {{\sigma {y_k}}}{{{{ | y  |}^2}}}( {{\varphi _k} *
    {\eta _{{\varepsilon _j}}}(y)} )} }  \right)dy} } }  \right|\\
    &\le  \left| {  \int_\Omega  {u(y)\mathrm{div}({f_0}({\eta _{{\varepsilon _0}}} *
    \varphi ))dy}  - \int_\Omega  {u(y){f_0}(y) \left( {\sum\limits_{k = 1}^d
    {\frac{{\sigma {y_k}}}{{{{ | y  |}^2}}}( {{\varphi _k} *
    {\eta _{{\varepsilon _0}}}(y)} )} } \right)dy} }  \right|\\
    &+  \left| {   \sum\limits_{j = 1}^\infty  {\int_\Omega  {u(y)\mathrm{div}
    ({f_j}({\eta _{{\varepsilon _j}}} * \varphi ))dy} }  - \sum\limits_
    {j = 1}^\infty  {\int_\Omega  {u(y){f_j}(y) \left({\sum\limits_{k = 1}^d
    {\frac{{\sigma {y_k}}}{{{{ | y  |}^2}}}( {{\varphi _k} *
    {\eta _{{\varepsilon _j}}}(y)} )} }  \right)dy} } }  \right|\\
    &\le | {{\nabla _{{{\mathcal H}_a}}}u}  |(\Omega ) +
    \sum\limits_{j = 1}^\infty  { | {{\nabla _{{{\mathcal H}_a}}}u}
     |({U_j})} \\
    &\le  | {{\nabla _{{{\mathcal H}_a}}}u}  |(\Omega ) +
     | {{\nabla _{{{\mathcal H}_a}}}u}  |(\Omega \backslash
    {\Omega _0})\\
    &\le  | {{\nabla _{{{\mathcal H}_a}}}u}  |(\Omega ) +
    3\varepsilon,
\end{align*} where we have used the fact (\ref{eq-2.2}) in the last
inequality. Note that $\psi (x) = {x_k}{ | x  |^{ - 2}}$, ${ \|
\varphi \|_{{L^\infty }}} \le 1$ and ${\rm{supp}}(\eta ) \subseteq
B(0,1)$. When $ | {x - y}  | < {\varepsilon _j} < {{ | y
|}}/{2}$, by the mean value theorem of multivariate functions, there
exists $\theta \in (0,1)$ such that
\[ | {\psi (x) - \psi (y)}  | = { | {y + \theta(x - y)}  |^{ - 2}}
 | {x - y}  | \le C | {x - y}  |{ | y  |^{ - 2}}.
\] Then, we obtain
\begin{align*}
{J_2}&=  \left| -{\sum\limits_{j = 0}^\infty  {\int_\Omega  {\int_\Omega
{\frac{1}{{\varepsilon _j^d}}\eta ( {\frac{{x - y}}{{{\varepsilon _j}
        }}} )u(y){f_j}(y) \left\{ {\sum\limits_{k = 1}^d {\left( {\frac{{
        \sigma {x_k}}}{{{{ | x  |}^2}}} - \frac{{\sigma {y_k}}}{{{{
         | y  |}^2}}}} \right){\varphi _k}(x)} }  \right\}dy} dx} } }  \right|\\
        &\le  |\sigma|\sum\limits_{j = 0}^\infty  {\int_\Omega
        {\int_\Omega  { \left| {\frac{1}{{\varepsilon _j^d}}\eta ( {\frac{
        {x - y}}{{{\varepsilon _j}}}} )u(y){f_j}(y)}  \right| \left\{ {\sum
        \limits_{k = 1}^{d} \frac{|x - y|}{|y + \theta(x - y)|^{2}} | {{\varphi _k}(x)}  | }  \right\}dy} dx} } \\
        &\le  |\sigma|\sum\limits_{j = 0}^\infty  {\int_\Omega
        {\int_\Omega  { \left|\frac{1}{{\varepsilon _j^d}}\eta ( {\frac{
        {x - y}}{{{\varepsilon _j}}}} ) \right|\frac{|x - y|}{|y + \theta(x - y)|^{2}}dx}|u(y)||f_j(y)|dy} } \\
        &\le C | \sigma   |\sum\limits_{j = 0}^\infty  {\int_\Omega
        {\int_\Omega  { \left| {\frac{1}{{\varepsilon _j^d}}\eta ( {\frac{{x - y}
        }{{{\varepsilon _j}}}} ) | {x - y}  |}  \right|dx}|
        {u(y)}  | | {{f_j}(y)}  |{{ | y  |}^{ - 2}}dy} } \\
        &\le C{\varepsilon _j} | \sigma   |\int_{{\mathbb{R}^d}} {|\eta (z)|dz}
        \sum\limits_{j = 0}^\infty  {\int_\Omega  { | {u(y)}  | | {
        {f_j}(y)}  |{{ | y  |}^{ - 2}}dy} } \\
        &= C{\varepsilon _j} | \sigma   |\int_{{\mathbb{R}^d}} {|\eta (z)|dz}
         {\int_\Omega  { | {u(y)}  |\Big| \sum\limits_{j = 0}^\infty {
        {f_j}(y)} \Big|{{ | y  |}^{ - 2}}dy} } \\
                &\le C{\varepsilon _j} | \sigma   |\int_{{\mathbb R^d}}
        {\eta (z)dz}  \\
       &\lesssim \varepsilon,
\end{align*} where we have used the
condition (\ref{eq4}).  By taking the supremum over $\varphi $ and
the arbitrariness of $\varepsilon>0$, the theorem can be proved.
\end{proof}

\begin{remark}
Suppose $\Omega$ is a bounded domain in $\mathbb{R}^d$ and the
weighted Sobolev space $W^{k,p}_0(\Omega,|x|^\alpha)$ is the closure
of the set $C^\infty_c(\Omega)$ in the norm
$$\|u\|:=\sum_{|\beta|\le k}\| \partial^\beta u\|_{L^{p}(|x|^\alpha)}.$$
with $1<p<{d}/{k}$ and $-2\le \alpha\le 0$. It is noted that
 $W^{k,p}_0(\Omega,|x|^\alpha)$ is exactly the
classical Sobolev space when $\alpha= 0$. The condition { (\ref
{eq4})} can be satisfied in one of the following three cases:
\item{\rm (1)}  $u\in W^{1,p}_0(\Omega,|x|^\alpha)$ for $1<p< -{d}/{\alpha}$ with $-2\le \alpha \le -1$ and $d\ge 3$;
\item{\rm (2)}   $u\in W^{2,2}_0(\Omega)$ for $d>4$;
\item{\rm (3)}  $u$ is any bounded function on $\Omega$ for $d\ge 3$.
\end{remark}
\begin{proof}Firstly, if  $u\in W^{1,p}_0(\Omega)$, then we use the
Hardy-Sobolev inequality in \cite[p. 259]{CKN84}   to obtain
\begin{align*}\int_{\Omega}|u(y)||y|^{-2}dy&\le |\Omega|^{{1}/{r'}}
\left(\int_{\Omega}|u(y)|^{r}|y|^{-2r}dy\right)^{{1}/{r}}\\
& \lesssim |\Omega|^{{1}/{r'}} \left(\int_{\Omega}|\nabla
u|^pdy\right)^{{1}/{p}}\\
&<\infty,
\end{align*}
where ${1}/{r}-{(1+\alpha)}/{d}={1}/{p}$ and
$1<r<{d}/{2}$. Secondly, if $u\in W^{2,2}_0(\Omega)$, then by
the Hardy-Sobolev inequality in
 \cite[(1.2)]{PR14} we obtain
\begin{align*}\int_{\Omega}|u(y)||y|^{-2}dy&\le |\Omega|^{{1}/{2}}
\left(\int_{\Omega}|u(y)|^{2}|y|^{-4}dy\right)^{{1}/{2}}\\
& \lesssim |\Omega|^{{1}/{2}} \left(\int_{\Omega}|\Delta
u|^2dy\right)^{{1}/{2}}\\
&<\infty.
\end{align*}Finally, since $u$ is bounded on $ \Omega$, then there exists a nonnegative constant $C$ such that
$|u(x)|\le C$ for any $x\in\Omega$. Moreover,  there exists a ball
$B(0, r)$ such that $E\subseteq B(0, r) $, where  $B(0,r)$ denotes a
ball with center $0$ and radius $r
> 0$. Then,  since $d\ge 3,$
\begin{align*}\int_{\Omega}|u(y)||y|^{-2}dy
 \le C  \left(\int_{B(0,r)}|y|^{-2}dy\right)^{{1}/{2}}
 <\infty.
\end{align*}
All in all, the condition { (\ref{eq4})} can be satisfied under the
above cases.
\end{proof}

Moreover, by Lemma \ref{semicontinuity} and Theorem \ref{th2.5},
we have the following max-min property of the
${\mathcal{H}_a}$-variation.

\begin{theorem}\label{th-2.1}
    Let   $\Omega\subset{\mathbb{R} ^d}$ be an open and bounded domain.  Suppose
    $u,v\in {L^1}(\Omega )$ obey   the condition (\ref{eq4}),
    then
   \[ | {{\nabla _{{{{\mathcal H}}_a}}}\max  \{ {u,v}  \}}
    |(\Omega ) +  | {{\nabla _{{{{\mathcal H}}_a}}}\min
   \{ {u,v}  \}}  |(\Omega ) \le  | {{\nabla _{{{{\mathcal
   H}}_a}}}u}  |(\Omega ) +  | {{\nabla _{{{{\mathcal H}}_a}}}v}
    |(\Omega ).\]
\end{theorem}

\subsection{Basic properties of ${\mathcal{H}_a}$-perimeter}\label{sec-2.2}
\hspace{0.6cm}
The ${{\mathcal H}_a}$-perimeter of $E\subseteq \Omega$ can be defined
as follows:
\begin{equation}\label{def-hp}
    P_{{\mathcal H}_a}(E,\Omega)= | {{\nabla _{{{{\mathcal H}}_a}}}
    {1_E}}  |(\Omega )=\sup_{\varphi\in \mathcal {F}(\Omega ) } \Big\{\int_E \mathrm{div}_{{\mathcal
    H}_a}\varphi(x)dx \Big\}.
\end{equation}

The following conclusion is a direct corollary of Lemma \ref {semicontinuity}.

\begin{corollary}(Lower semicontinuity of ${P_{{{{\mathcal H}}_a}}}$)
                  \label{semicontinuity-2}
  Suppose ${1_{{E_k}}} \to {1_E}$ in $L_{loc}^1(\Omega )$, where $E$ and
  ${E_k}$, $k \in \mathbb N,$ are subsets of $\Omega $, then \[{P_{{{{
  \mathcal H}}_a}}}(E,\Omega ) \le \mathop {\lim \inf }\limits_{k \to
  \infty } {P_{{{{\mathcal H}}_a}}}({E_k},\Omega ).\]
\end{corollary}

Moreover, by Theorem \ref {th-2.1}, via choosing $u=1_E$ and $v=1_F$
for any compact subsets $E, F$ in $\Omega$, we immediately obtain
the following corollary. Moreover, it follows from \cite[Section 1.1
(iii)]{XZ} that the  equality condition of (\ref{eq10})  can be
similarly obtained.

\begin{corollary}
    For any compact subsets $E, F$ in $\Omega$, we have
    \begin{equation}\label{eq10}P_{{\mathcal H}_a}(E\cap F,\Omega)+P_{{\mathcal H}_a}(E\cup F,
    \Omega)\le P_{{\mathcal H}_a}(E,\Omega)+P_{{\mathcal H}_a}
    (F,\Omega).\end{equation}
Especially, if $P_{{\mathcal H}_a}(E\setminus (E\cap F),\Omega)
\cdot P_{{\mathcal H}_a}(F\setminus (F\cap E),\Omega)=0$, the
equality of (\ref{eq10}) holds true.  \end{corollary}

\begin{proof}Since (\ref{eq10}) is valid, we only need to prove its converse
inequality  holds true under the above condition.  Obviously, the
condition $P_{{\mathcal H}_a}(E\setminus (E\cap F),\Omega) \cdot
P_{{\mathcal H}_a}(F\setminus (F\cap E),\Omega)=0$ implies that
$P_{{\mathcal H}_a}(E\setminus (E\cap F),\Omega)=0$ or $P_{{\mathcal
H}_a}(F\setminus (F\cap E),\Omega)=0$. Suppose $P_{{\mathcal
H}_a}(E\setminus (E\cap F),\Omega)=0$. Via (\ref{eq10}), we have
\begin{eqnarray}\label{eq11}
P_{{\mathcal H}_a}(E,\Omega)&=&P_{{\mathcal H}_a}((E\setminus (E\cap
F))\cup (E\cap F),\Omega)\\ \nonumber &\le& P_{{\mathcal
H}_a}(E\setminus (E\cap F),\Omega)+P_{{\mathcal H}_a}( E\cap
F,\Omega)\\ \nonumber  &=&P_{{\mathcal H}_a}( E\cap F,\Omega).
\end{eqnarray}
Using (\ref{def-hp}) and $E\cup F=F\cup (E\setminus (E\cap F))$, we
obtain
\begin{eqnarray}\label{eqq12}
P_{{\mathcal H}_a}(F,\Omega)&=&\sup_{\varphi\in \mathcal {F}(\Omega
) } \Big\{\int_F \mathrm{div}_{{\mathcal
    H}_a}\varphi(x)dx \Big\}\\ \nonumber
&=& \sup_{\varphi\in \mathcal {F}(\Omega ) } \Big\{\int_{E\cup F}
\mathrm{div}_{{\mathcal
    H}_a}\varphi(x)dx -\int_{E\setminus (E\cap F)}
\mathrm{div}_{{\mathcal
    H}_a}\varphi(x)dx \Big\}\\ \nonumber
    &\le& \sup_{\varphi\in \mathcal {F}(\Omega ) } \Big\{\int_{E\cup F}
\mathrm{div}_{{\mathcal
    H}_a}\varphi(x)dx\Big\} +\sup_{\varphi\in \mathcal {F}(\Omega ) } \Big\{\int_{E\setminus (E\cap F)}
\mathrm{div}_{{\mathcal
    H}_a}\varphi(x)dx \Big\}\\ \nonumber
    &=&P_{{\mathcal
H}_a}( E\cup F,\Omega)+ P_{{\mathcal H}_a}(E\setminus E\cap
F,\Omega)\\ \nonumber &=&P_{{\mathcal H}_a}( E\cup F,\Omega).
\end{eqnarray}
Combining (\ref{eq11}) with (\ref{eqq12}) deduces that
\begin{eqnarray*}
P_{{\mathcal H}_a}(E,\Omega)+P_{{\mathcal H}_a}(F,\Omega)\le
P_{{\mathcal H}_a}( E\cup F,\Omega)+P_{{\mathcal H}_a}( E\cap
F,\Omega),
\end{eqnarray*}which derives the desired result. Another case can be
similarly proved, we omit the details.
\end{proof}

For the ${{{\mathcal H}}_a}$-variation and the ${{{\mathcal
H}}_a}$-perimeter, we can prove a coarea inequality   for functions
in ${{\mathcal B}}{{{\mathcal V}}_{{{\mathcal H}}_a}}(\Omega )$.
\begin{theorem}\label{th-2.9}
If $f \in {{\mathcal B}}{{{\mathcal V}}_{{{\mathcal H}}_a}}(\Omega )$,
then
\[ | {{\nabla _{{{{\mathcal H}}_a}}}f}  |(\Omega ) \le
\int_{ - \infty }^{ + \infty } {{P_{{{{\mathcal H}}_a}}}({E_t},
\Omega )dt}, \]where ${E_t}=\{x\in\Omega: f(x)>t\}$ for
$t\in\mathbb{R}$.
\end{theorem}
\begin{proof}
Let $\varphi\in C^1_c(\Omega,\mathbb{R}^{d})$ and ${ \| \varphi
\|_{{L^\infty }(\Omega)}} \le 1$. We can easily prove that for $i =
1,2,\ldots,d$,
$$-\int_\Omega  {f(x)\frac{{\sigma {x_i}}}{{{{ | x  |}^2}}}
\varphi_i (x)dx}  = -\int_{ - \infty }^{ + \infty } \Big(
{\int_{{E_t}} {\frac{{\sigma {x_i}}}{{{{ | x  |}^2}}}\varphi_i(x)
dx} } \Big)dt
$$
 and
$$\int_{\Omega}f (x)\mathrm{div} \varphi(x) dx=\int^\infty_{-\infty}
\Big(\int_{E_t}\mathrm{div} \varphi (x)dx\Big)dt, $$ where the
latter can be seen in the proof of \cite[Section 5.5, Theorem
1]{EG}.

It follows that $$\int_{\Omega}f(x)\mathrm{div}_{\mathcal{H} _a}
\varphi(x) dx=\int^\infty_{-\infty}\Big(\int_{E_t}\mathrm{div}_
{\mathcal{H} _a}\varphi(x) dx\Big)dt.$$ Therefore, we conclude that
for all $\varphi$ as above,
 \[\int_\Omega  {f(x)\mathrm{div}{_{{{{
\mathcal H}}_a}}}\varphi(x) dx}  \le \int_{ - \infty }^{ + \infty }
{{P_{{{{\mathcal H}}_a}}}({E_t},\Omega )dt}. \] Furthermore,
\[ | {{\nabla _{{{{\mathcal H}}_a}}}f}  |(\Omega )
\le \int_{ - \infty }^{ + \infty } {{P_{{{{\mathcal H}}_a}}}
({E_t},\Omega )dt}. \]
\end{proof}

\section{$\mathcal{\widetilde{H}}_{\sigma}$-BV functions} \label{sec-3}
\hspace{0.6cm} In order to overcome the deficiency of ${{\mathcal
B}}{{{\mathcal V}}_{{{\mathcal H}}_a}}(\Omega )$   in the coarea
formula, we turn to study another form of  divergence operator and
gradient operator related to the operator
$\mathcal{\widetilde{H}}_{\sigma}$ which is closely related to the
operator ${{\mathcal H}}_a$. Via a simple computation, we  obtain,
for $u\in C_c^2(\Omega)$,
\begin{align*}
    \mathcal{\widetilde{H}}_{\sigma}u:=-\mathrm{div}{_{\mathcal{\widetilde{H}}_
    {\sigma}}}( {{\nabla _{\mathcal{\widetilde{H}}_{\sigma}}}u } )
    &=-\Big( {   \nabla  - \sigma \frac{x}{{{{ | x  |}^2}}},
      \nabla  + \sigma \frac{x}{{{{ | x  |}^2}}}}\Big)
    \cdot \Big( {\nabla  + \sigma \frac{x}{{{{ | x  |}^2}}},
    \nabla  - \sigma \frac{x}{{{{ | x  |}^2}}}}\Big)u \\
    & = -\Big( {   \nabla  - \sigma \frac{x}{{{{ | x  |}^2}}},
      \nabla  + \sigma \frac{x}{{{{ | x  |}^2}}}} \Big) \cdot
    \Big( {( {\nabla  + \sigma \frac{x}{{{{ | x  |}^2}}}}
    )u ,( {\nabla  - \sigma \frac{x}{{{{ | x  |}^2}}}}
    )u } \Big)\\
    & = -( {  \nabla  -\sigma \frac{x}{{{{ | x  |}^2}}}}
    )( {\nabla  + \sigma \frac{x}{{{{ | x  |}^2}}}}
    )u  - ( { \nabla + \sigma \frac{x}{{{{ | x  |
    }^2}}}} )( {\nabla  - \sigma \frac{x}{{{{ | x  |}
    ^2}}}} )u \\
    &= \Big( { - \Delta  - \sigma ( {d - 2 - \sigma } )\frac{1}
    {{{{ | x  |}^2}}}} \Big)u  + \Big( { - \Delta  + \sigma
    ( {( {d - 2 - \sigma } )\frac{1}{{{{ | x  |}^2}}} +
    \frac{{2\sigma }}{{{{ | x  |}^2}}}} )} \Big)u \\
    &= 2\Big( { - \Delta  + \frac{{{\sigma ^2}}}{{{{ | x  |}^2}}}}\Big)u.
\end{align*}
 %It should be noted that when $a=0$,
%$ \mathrm{div}{_{\mathcal{\widetilde{H}}_{\sigma}}}(   \nabla
%_{\mathcal{\widetilde{H}}_{\sigma}}u   )=-2 \Delta u=2{{{\mathcal
%H}}_{{0}}}u$.

The $\mathcal{\widetilde{H}}_{\sigma}$-divergence of a vector valued
function \[\Phi  = ({\varphi _1}, {\varphi _2}, \ldots ,{\varphi
_{2d}}) \in C_c^1(\Omega, {\mathbb{R} ^{2d}})\] is defined as
\[\mathrm{div}_{\mathcal{\widetilde{H}}_{\sigma}}\Phi :={A_{ - 1,a}}
{\varphi _1} +  \cdots  + {A_{ - d,a}}{\varphi _d} + {A_{
1,a}}{\varphi _ {d + 1}} +  \cdots  + {A_{ d,a}}{\varphi _{2d}}. \]
For $u\in C_c^1(\Omega)$, the
$\mathcal{\widetilde{H}}_{\sigma}$-gradient of $u$ is defined as
\[{\nabla _ {\mathcal{\widetilde{H}}_{\sigma}}}u := ({A_{1,a}}u, \ldots
,{A_{d,a}}u,  {A_{ - 1,a}}u, \ldots , {A_{ - d,a}})u. \]

Let $\Omega \subseteq \mathbb{R}^{d}$ be an open set. The
${\mathcal{\widetilde{H}}_{\sigma}}$-variation of $f \in
{L^1}(\Omega )$ is defined by \[ | {{\nabla
_{{\mathcal{\widetilde{H}}_{\sigma}}}}f}
 |(\Omega ) = \mathop {\sup } \limits_{\Phi  \in
{\widetilde{\mathcal F}(\Omega )}}  \Big\{ {\int_\Omega  {f(x)
\mathrm{div}{_{{\mathcal{\widetilde{H}}_{\sigma}}}}\Phi (x)dx} }
 \Big\},\] where $\widetilde{{\mathcal F}}(\Omega )$ denotes the
class of all functions
\[\Phi  = ( {{\varphi _1},{\varphi _2}, \ldots ,{\varphi _
{2d}}} ) \in C_c^1(\Omega, {\mathbb R^{2d}}) \] satisfying
\[{ \| \Phi    \|_\infty } = \mathop {\sup }\limits_
{x \in \Omega } \Big\{( {{{ | {{\varphi _1}(x)}  |}^2} + \ldots + {{ |
{{\varphi _{2d}}(x)}  |}^2}} )^{{1}/{2}}\Big\} \le 1. \]

An function ${f\in L^1}(\Omega )$ is said to have the
${\mathcal{\widetilde{H}}_{\sigma}}$-bounded variation on $\Omega $
if
\[ | {{\nabla _ {{\mathcal{\widetilde{H}}_{\sigma}}}}f}
 |(\Omega ) < \infty ,\] and the collection of all such
functions is denoted by
$\mathcal{BV}_{{\mathcal{\widetilde{H}}_{\sigma}} }   ( \Omega
  ) $, which is a Banach space with the norm
\[ { \| f  \|_{{{\mathcal B}}{{{\mathcal V}}_{{\mathcal{\widetilde{H}}
_{\sigma}}}}(\Omega )}} = { \| f  \|_{{L^{1}(\Omega)}}} +  |
{{\nabla _ {{\mathcal{\widetilde{H}}_{\sigma}}}}f}  |(\Omega ). \]

The ${\mathcal{\widetilde{H}}_{\sigma}}$-perimeter of $E\subseteq
\Omega$ can be defined as follows:
$$P_{\mathcal{\widetilde{H}}_{\sigma}}(E,\Omega)=\big| \nabla_
{\mathcal{\widetilde{H}}_{\sigma}} 1_E \big|(\Omega) =\sup_{\Phi\in
\mathcal {\widetilde{F}}(\Omega ) } \Big\{\int_E \mathrm
{div}_{{\mathcal{\widetilde{H}}_{\sigma}}}\Phi(x)dx\Big\}. $$
It is
easy to see that if $u$ belongs to ${{\mathcal B}}{{{\mathcal
V}}_{\mathcal{\widetilde{H}}_{\sigma}}} ( \Omega )$, then $u$ also
belongs to ${{\mathcal B}}{{ {\mathcal V}}_{{{\mathcal H}}_a}}(
\Omega  )$. In fact, this can be proved by choosing
$\Phi=(\varphi_1,\ldots,\varphi_d, 0, \ldots,0)$ in the definition
of the ${\mathcal{\widetilde{H}}_{\sigma}}$-variation of $u$.

\subsection{Basic properties of $\mathcal{B} {\mathcal{V} _{\mathcal
            {\widetilde{H}}_{\sigma}}}(\Omega )$}
\hspace{0.5cm}

In this subsection,   using similar methods, we conclude that $\mathcal{B} {\mathcal{V}
_{\mathcal{\widetilde{H}}_{\sigma}}}(\Omega )$  enjoys several
properties as same as those of $\mathcal{B} {\mathcal{V} _{{ \mathcal H}_a}}(\Omega
)$. For convenience, we list the following
results for $\mathcal{B} {\mathcal{V}
_{\mathcal{\widetilde{H}}_{\sigma}}}(\Omega )$ and omit the
proofs.

\begin{lemma}\label{function-1}
  Let $u \in {{\mathcal B}}
 {{{\mathcal V}}_{{\mathcal{\widetilde{H}}_{\sigma}}}}(\Omega )$. There exists
  a unique $\R^{2d}$-valued finite Radon measure ${D_{{{\mathcal{\widetilde{H}}}}_\sigma}}u=
  (D_{{A_{1,a}}}u,\ldots,D_{{A_{d,a}}}u,D_{{A_{-1,a}}}u,\ldots,D_{{A_{-d,a}}}u)$
  such that \[\int_\Omega  u (x)\mathrm{div}{_{\mathcal
  {\widetilde{H}} _\sigma}}\Phi (x)dx = \int_\Omega  \Phi  (x) \cdot d{D_{{{\mathcal
  {\widetilde{H}}}}_\sigma}}u\] for every $\Phi\in C^{\infty}_c(\Omega,
  \R^{2d})$ and \[\big| {\nabla
  _{\mathcal{\widetilde{H}}_{\sigma}}}u\big|
  (\Omega ) = |{D_{{{\mathcal{\widetilde{H}}}}_\sigma}}u|(\Omega ), \] where
  $|{D_{{{\mathcal{\widetilde{H}}}}_\sigma}}u|$ is the total variation of
  the measure ${D_{{{\mathcal{\widetilde{H}}}}_\sigma}}u$.
\end{lemma}

Similar  to Definition \ref{def-Sobolev}, we give the definition
of the Sobolev space associated with ${\widetilde{\mathcal
H}_\sigma}$.

\begin{definition}\label{def-Sobolev1}
    Suppose $\Omega$ is an open set in $\mathbb{R} ^{d}$ for $d \ge 2$. Let $1 \le p \le
    \infty$. The Sobolev space $W_{{\mathcal{\widetilde{H}}_{\sigma}}}^{k,p}
    (\Omega )$ is defined as the set of all functions $f \in {L^p}(\Omega )$
    such that \[{A_{{j_1},a}}\cdots {A_{{j_m},a}}f \in {L^p}({\Omega}
    ),\quad 1 \le  | {{j_1}}  |,\ldots ,  | {{j_m}}  |\le d,1
    \le m \le k.\] The norm of $f \in W_{{\mathcal{\widetilde{H}}_{\sigma}}}^
    {k,p}(\Omega )$ is defined as \[{ \| f  \|_{W_{{{\cal \widetilde
    {H}}_\sigma}}^{k,p}}}: = \sum\limits_ {1 \le  | {{j_1}}  |,
    \ldots ,  | {{j_m}}  |\le d} {{{ \| {{A_{{j_1},a}}
    \cdots {A_{{j_m},a}}f}  \|}_{{L^p}}}}  + { \| f
    \|_{{L^p}}},\]
    where $a \ge  - {( {{{\frac{d}{2} - 1}}} )^2}$.
\end{definition}

It follows from Definition \ref{def-Sobolev1} that
$W_{{\mathcal{\widetilde{H}}_{\sigma}}}^{k,p}(\Omega )\subseteq
W_{{{\mathcal H}_a}}^{k,p}(\Omega )$.

%\begin{lemma}
%    The space $(\mathcal{BV}_{\mathcal{\widetilde{H}}_{\sigma}}(\Omega),|\cdot|_{
%    \mathcal{BV}_{\mathcal{\widetilde{H}}_{\sigma}}(\Omega)})$ is a Banach space.
%\end{lemma}

\begin{lemma}
    \item{{\rm (i)}} Suppose that $f \in W_{{{\cal \widetilde{H}}_\sigma}}^{1,1}(\Omega )$. Then
    \begin{equation*}
         | {{\nabla _{{\mathcal{\widetilde{H}}_{\sigma}}}}f}  |(\Omega ) = \int_
        \Omega  { | {{\nabla _{\mathcal{\widetilde{H}}_{\sigma}}}f(x)}  |dx}.
    \end{equation*}
    \item{{\rm (ii)}} Suppose that ${f_k} \in {{\mathcal B}}{{{\mathcal V}}_
        {{\mathcal{\widetilde{H}}_{\sigma}}}}(\Omega ), k\in \mathbb N $ and ${f_k} \to f$ in $L^{1}_{loc}(\Omega)$.
    Then
    \[ | {{\nabla _{{\mathcal{\widetilde{H}}_{\sigma}}}}f}  |(\Omega ) \le \mathop
    {\lim \inf }\limits_{k \to \infty }  | {{\nabla _{{\mathcal{\widetilde{H}}_{\sigma}}}}
        {f_k}}  |(\Omega ). \]
\end{lemma}

\begin{theorem}\label{th-3.3}
    Let $\Omega\subset{\mathbb{R} ^d}$ be an open and bounded domain. Assume that $u \in {{\mathcal B}}
    {{{\mathcal V}}_{{\mathcal{\widetilde{H}}_{\sigma}}}}(\Omega )$  satisfying
    the condition (\ref{eq4}), then there exists a sequence
    ${ \{ {{u_h}}  \}_{h \in \mathbb{N} }} \in {{
    \mathcal B}}{{{\mathcal V}}_{{\mathcal{\widetilde{H}}_{\sigma}}}}(\Omega ) \cap C_c^
    \infty (\Omega ) $ such that $$\mathop {\lim }\limits_{h \to \infty }
    { \| {{u_h} - u}  \|_{{L^1}}} = 0$$ and \[\mathop {\lim }
    \limits_{h \to \infty } \int_\Omega  { | {{\nabla _{{{\mathcal {\widetilde{H}}}}_
   \sigma}}{u_h}(x)}  |dx}  =  | {{\nabla _{\mathcal{\widetilde{H}}_{\sigma}}}u}
     |(\Omega ). \]
\end{theorem}

\begin{theorem}
    Let   $\Omega\subset{\mathbb{R} ^d}$ be an bounded open domain.  Suppose $u,v\in {L^1}(\Omega )$ and  satisfy the condition (\ref{eq4}),
    then
 \[ | {{\nabla _{{\mathcal{\widetilde{H}}_{\sigma}}}}\max  \{ {u,v}  \}}
 |(\Omega ) +  | {{\nabla _{{\mathcal{\widetilde{H}}_{\sigma}}}}\min  \{ {u,v}  \}}
  |(\Omega ) \le  | {{\nabla _{{\mathcal{\widetilde{H}}_{\sigma}}}}u}
 |(\Omega ) +  | {{\nabla _{{\mathcal{\widetilde{H}}_{\sigma}}}}v}  |(\Omega ).\]
\end{theorem}

\begin{corollary}(Lower semicontinuity of $P_{\mathcal{\widetilde{H}}_{\sigma}}$)
    \label{semicontinuity-3}
    Suppose ${1_{{E_k}}} \to {1_E}$ in $L_{loc}^1(\Omega )$, where $E$ and
    ${E_k}$, $k \in \mathbb N,$ are subsets of $\Omega $, then \[P_{\mathcal{\widetilde{H}}_{\sigma}}(E,\Omega)\le \mathop {\lim \inf }\limits_{k \to \infty } P_{\mathcal{\widetilde{H}}_{\sigma}}({E_k},\Omega).\]
\end{corollary}

\begin{corollary}\label{coro2.5}
    For any compact subsets $E, F$
    in $\Omega$, we have \begin{equation}\label{eqq13}P_{\mathcal{\widetilde{H}}_{\sigma}}(E\cap F,\Omega)
    +P_{\mathcal{\widetilde{H}}_{\sigma}}(E\cup F,\Omega)\le P_{\mathcal
    {\widetilde{H}}_{\sigma}}(E,\Omega)+P_{\mathcal{\widetilde{H}}_{\sigma}}(
    F,\Omega).\end{equation}
Especially, if $P_{\mathcal{\widetilde{H}}_{\sigma}}(E\setminus
E\cap F,\Omega) \cdot
P_{\mathcal{\widetilde{H}}_{\sigma}}(F\setminus (F\cap
E),\Omega)=0$, the equality of (\ref{eqq13}) holds true.
\end{corollary}

The following lemma gives a scaling relation of the
$\mathcal{\widetilde{H}}_{\sigma}$-perimeter   in $\Omega$.

\begin{lemma}\label{coro2.52}
For any set $E$ in $\Omega$, denote by $sE$ the set $\{sx: x\in
E\}$. If $sE\subseteq \Omega$ for  $s>0$, then
\begin{equation*}\label{equaq3.8}
P_{\mathcal{\widetilde{H}}_{\sigma}}(sE, \Omega)=
s^{d-1}P_{\mathcal{\widetilde{H}}_{\sigma}}(E, \Omega).
\end{equation*}
\end{lemma}

\begin{proof}
By the definition of the
$\mathcal{\widetilde{H}}_{\sigma}$-perimeter, we have
%\begin{eqnarray*}
%P_{\mathcal{H}_{\alpha}}(sE)&=& \sup_{\varphi\in \mathcal
%F}(\Omega)}\Big\{\int_{sE}
%\mathrm{div}_{\mathcal{\widetilde{H}}_{\sigma}}\varphi(x)dx\Big\}\\
%&=&\sup_{\varphi\in \mathcal
%{F}(\Omega)}\Big\{\int_{E}s^{d-1}\Big[\sum^d_{k=1}-\frac{\partial}{\partial
%x_k}\varphi_{d-k+1}(sx)-\sum^d_{k=1} \frac{\partial}{\partial
%x_k}\varphi_{d+1+k}(sx)\Big]\\&&+s^{d-1}\Big[\sum^d_{k=1}
% x_k|x|^{-2}\varphi_{d-k+1}(sx)-\sum^d_{k=1}
% x_k|x|^{-2}\varphi_{d+1+k}(sx)\Big]dx\Big\}.
%\end{eqnarray*}

%Then via the definition of $\mathcal{\widetilde{H}}_{\sigma}$-perimeter,
%we have
%$$P_{\mathcal{H}_{\alpha}}(sE)
%=s^{d-1}P_{\mathcal{H}_{\alpha}}(E).$$
\begin{eqnarray*}
    P_{\mathcal{\widetilde{H}}_{\sigma}}(sE, \Omega)
    &=& \sup_{\Phi\in {\widetilde{{\mathcal F}}(\Omega )}}\Big\{\int_{sE}
    \mathrm{div}_{\mathcal{\widetilde{H}}_{\sigma}}\Phi(x)dx\Big\}\\
    &=&\sup_{\Phi\in {\widetilde{{\mathcal F}}(\Omega )}}\Bigg\{\int_{E}s^{d-1}\Big(\sum^d_{k=1} \partial_{x_k} \varphi_{k}(sx)+\sum^d_{k=1}
    \partial_{x_k}\varphi_{d+k}(sx)\Big)\\
    &&-s^{d-1}\Big(\sum^d_{k=1}
    x_k|x|^{-2}\varphi_{k}(sx)-\sum^d_{k=1}
    x_k|x|^{-2}\varphi_{d+k}(sx)\Big)dx\Bigg\}.
\end{eqnarray*}
Then
\begin{equation*}
    P_{\mathcal{\widetilde{H}}_{\sigma}}(sE, \Omega)=
    s^{d-1}P_{\mathcal{\widetilde{H}}_{\sigma}}(E, \Omega).
\end{equation*}
\end{proof}

An immediate corollary of Lemma \ref{coro2.52} is given as follows.

\begin{corollary}\label{coro2.51}
Let $B(x,s)$ be the open ball in $\Omega$ centered at $x$ with
radius $s$.
\begin{equation*}\label{equa3.8}
P_{\mathcal{\widetilde{H}}_{\sigma}}(B(x,s),
\Omega)=P_{\mathcal{\widetilde{H}}_{\sigma}}(B(x,1), \Omega)
s^{d-1}.
\end{equation*}
\end{corollary}

\begin{remark}\label{rem2.13} It should be noted that    the set $E$ and its
complementary set have the same  perimeter in the classical case.
But unfortunately, for the case of the
$\mathcal{\widetilde{H}}_{\sigma}$-perimeter, the above fact does
not hold. For example, let $\Omega=\mathbb{R}^d\backslash\{0\}$ and
$E=B(x,r)$ with $r>0$. The definition of the
$\mathcal{\widetilde{H}}_{\sigma}$-perimeter and Corollary \ref
{coro2.51} indicate  that
\begin{equation*}\label{eq-converse-example}
P_{\mathcal{\widetilde{H}}_{\sigma}}(B(x,r)^c,\Omega)\gtrsim\int_{B(x,r)^c}|y|^{-1}dy
=\infty>P_{\mathcal{\widetilde{H}}_{\sigma}}(B(x,r),\Omega).
\end{equation*}
\end{remark}

\subsection{Coarea formula and Sobolev's  inequality of  $\mathcal
           {\widetilde{H}}_{\sigma}$-BV functions}
\hspace{0.6cm}
 In what follows,  we prove the coarea formula for $\mathcal{\widetilde{H}
}_{\sigma}$-BV functions. Before proving this result, we give the
following lemma.
\begin{lemma}\label{le-3.1}
    If $f \in C^{1}(\Omega )$, then
    \begin{equation}\label{eq-3.11}
        | \nabla f (x) | + \frac{ | \sigma   |}{ | x  |} | {f(x)}  | \le   | {{\nabla _{\mathcal{\widetilde{H}}_
        {\sigma}}}f(x)} | \le \sqrt{2} \Big( { | {\nabla f(x)}  |
        + \frac{{ | \sigma  |}}{{ | x  |}} | {f(x)}
         |} \Big).
    \end{equation}
    %and
    %\begin{equation}\label{eq-3.12}
         %| \nabla f (x) | + \frac{ | \sigma   |}{ | x  |} | {f(x)}  | \le   | {{\nabla _{\mathcal{\widetilde{H}}_{\sigma}}}f(x)}|.
    %\end{equation}
\end{lemma}

\begin{proof}
    By the definition of $ {{\nabla _{\mathcal{\widetilde{H}}_{\sigma}}}f} $, we
    have
    \begin{align*}
        { | {{\nabla _{\mathcal{\widetilde{H}}_{\sigma}}}f}  |^2}
        &= { | {{A_{1,a}}f}  |^2} +  \cdots  + { | {{A_{d,a}}f}
             |^2} + { | { {A_{ - 1,a}}f}  |^2} +  \cdots  + { |
            {{A_{ - d,a}}f}  |^2}\\
        &= \sum\limits_{i= 1}^d {( {{{ | {{A_{i,a}}f}  |}^2} +
                {{ | {  {A_{ - i,a}}f}  |}^2}} )}\\
                &=\sum\limits_{i= 1}^d\Big| { \frac{\partial f}{\partial x_i}(x)}  +
       {\frac{{\sigma {x_i}}}{{{{ | x  |}^2}}}f(x)}  \Big|^2+\sum\limits_{i= 1}^d\Big| { \frac{\partial f}{\partial x_i}(x)}
       -
       {\frac{{\sigma {x_i}}}{{{{ | x  |}^2}}}f(x)}  \Big|^2\\
       &=2 \sum\limits_{i= 1}^d\Big| { \frac{\partial f}{\partial x_i}(x)}\Big|^2+2
       \sum\limits_{i= 1}^d\Big| {\frac{{\sigma {x_i}}}{{{{ | x  |}^2}}}f(x)}\Big|^2\\
       &=2 \sum\limits_{i= 1}^d\Big| { \frac{\partial f}{\partial x_i}(x)}\Big|^2+2
       {\frac{{\sigma^2  }}{{{{ | x  |}^4}}}|f(x)|^2}\big(\sum\limits_{i= 1}^d  { {{ }}{{{{ | x_i  |}^2}}} }\big)\\
       &=2 | \nabla f (x) |^2 +2 \frac{  \sigma^2  }{ | x  |^2} | {f(x)}  |^2.
    \end{align*}
Then it is easy to see that
\begin{equation*}
        | \nabla f (x) | + \frac{ | \sigma   |}{ | x  |} | {f(x)}  | \le   \sqrt{2}\Big( | \nabla f (x) |^2 +
         \frac{  \sigma^2  }{ | x  |^2} | {f(x)}  |^2 \Big)^{{1}/{2}}=  | {{\nabla _{\mathcal{\widetilde{H}}_
        {\sigma}}}f(x)} |\le \sqrt{2} \Big( { | {\nabla f(x)}  |
        + \frac{{ | \sigma  |}}{{ | x  |}} | {f(x)}
         |} \Big),
    \end{equation*}
which derives that (\ref {eq-3.11}) is valid.
\end{proof}

\begin{theorem}\label{coarea formula}
     Let $\Omega\subset{\mathbb{R} ^d}$ be a  bounded open domain. If $f \in {{\mathcal B}}{{{\mathcal V}}_{{\mathcal{\widetilde{H}}_{\sigma}}}}
     (\Omega )$ satisfying the condition (\ref{eq4}), then
 \begin{equation}\label{coarea formula-1}
     | {{
    \nabla _{{\mathcal{\widetilde{H}}_{\sigma}}}}f}  |(\Omega ) \approx \int_
    { - \infty }^{ + \infty } {{P_{{\mathcal{\widetilde{H}}_{\sigma} }}}({E_t},\Omega )
    dt},
 \end{equation}
where ${E_t}=\{x\in\Omega: f(x)>t\}$ for $t\in\mathbb{R}$.
\end{theorem}
\begin{proof}
Firstly, suppose \[\Phi  = ( {{\varphi _1},{\varphi _2}, \ldots
,{\varphi _ {2d}}} ) \in C_c^1(\Omega, {\mathbb R^{2d}}) .\] We can
easily prove that for $i = 1,2,\ldots,d $,
$$-\int_\Omega {f(x)\frac{{\sigma {x_i}}}{{{{ | x
 |}^2}}}\varphi_i(x) dx}  = -int_{ - \infty }^{ + \infty } {\Big(
{\int_{{E_t}} {\frac{{\sigma {x_i}}}{{{{  | x
 |}^2}}}\varphi_i(x) dx} } \Big)dt}, $$
 $$\int_\Omega
{f(x)\frac{{\sigma {x_i}}}{{{{ | x  |}^2}}}\varphi_{d+i}(x) dx} =
\int_{ - \infty }^{ + \infty } {\Big( \int_{{E_t}} {\frac{{\sigma
{x_i}}}{{{{ | x  |}^2}}}\varphi_{d+i}(x) dx}  \Big)dt}, $$
$$\int_{\Omega}f(x) \mathrm{div} (\varphi_1(x),\ldots,\varphi_d(x))
dx=\int^\infty_{-\infty}\Big(\int_{E_t}\mathrm{div}
(\varphi_1(x),\ldots,\varphi_d(x)) dx\Big)dt  $$ and
$$ \int_{\Omega}f (x)\mathrm{div} (\varphi_{d+1}(x),\ldots,\varphi_{2d}(x))
dx= \int^\infty_{-\infty}\Big(\int_{E_t}\mathrm{div}
(\varphi_{d+1},\ldots,\varphi_{2d}) dx\Big)dt, $$ where the latter
can be seen in the proof of \cite [Section 5.5, Theorem 1]{EG}. It
follows that
$$\int_{\Omega}f(x)\mathrm {div}_{\mathcal{\widetilde{H}}_{\sigma}}\Phi(x)
dx=\int^\infty_{-\infty}\Big(\int_{E_t}
\mathrm{div}_{\mathcal{\widetilde{H}} _\sigma}\Phi(x) dx\Big)dt.$$
Therefore, we conclude that for all $\Phi\in
\mathcal{\widetilde{F}}(\Omega)$,
\[\int_\Omega {f(x)\mathrm{div}{_{{{{ \mathcal
{\widetilde{H}}}}_\sigma}}}\Phi(x) dx} \le \int_{ - \infty }^{ +
\infty } {{P_{{{ {\mathcal {\widetilde{H}}}}_\sigma}}}({E_t},\Omega
)dt} .\] Furthermore, \[ | {{\nabla
_{{\mathcal{\widetilde{H}}_{\sigma}}}}f}
 |(\Omega ) \le \int_{ - \infty }^{ + \infty } {{P_{{{{\mathcal
{\widetilde{H}}}}_\sigma}}}({E_t},\Omega )dt}.
\]

Secondly, without loss of generality, we only need to verify that
\[ | {{ \nabla _{{\mathcal{\widetilde{H}}_{\sigma}}}}f}
 |(\Omega ) \ge \int_{ - \infty }^{ + \infty }
{{P_{{\mathcal{\widetilde{H}}_{\sigma}}}}({E_t},\Omega )dt}\] holds
for $f\in
\mathcal{BV}_{\mathcal{\widetilde{H}}_{\sigma}}(\Omega)\bigcap
C^\infty(\Omega)$. This proof can refer to the  idea of
\cite[Proposition 4.2]{M}. Let
\begin{equation*}
    \label{eqq2.8}m(t)=\int_{\{x\in\Omega:\ f(x)\le t\}}|\nabla  f(x)|dx.
\end{equation*}
It is obvious that
\begin{equation*}
    \int^\infty_{-\infty}m'(t)dt\le\int_{\Omega}|\nabla f(x)|dx.
\end{equation*}

Define the following function ${g_{_h}}$ as
$${g_{_h}}(s):=
\begin{cases}
    0, &\mathrm{if} \ s\le t,\\
    h(s-t), &\mathrm{if}\ t\le s\le t+1/h,\\
    1, &\mathrm{if} \   s\ge t+1/h,
\end{cases}$$
where $t\in \mathbb{R}$. Set the sequence ${v_h}(x): = {g_h}(f(x))$.
At this time, ${v_h} \to {1_{{E_t}}}$ in $L^1(\Omega)$. In fact,
\begin{align*}
    \int_\Omega  { | {{v_h}(x) - {1_{{E_t}}}}  |dx}
    &= \int_{\{ x \in \Omega :t < f(x) \le t + {1 \mathord{ /
                {\vphantom {1 h}}  .
                \kern-\nulldelimiterspace} h}\} } {{g_h}(f(x))dx} \\
    &\le  \Big| {\Big\{ x \in \Omega :t < f(x) \le t + {1 \mathord{ /
                {\vphantom {1 h}}
                \kern-\nulldelimiterspace} h}\Big\} }  \Big| \to 0.
\end{align*}
Since $\{x\in\Omega:\ t<f(x)\le t+1/h\} \rightarrow  \emptyset$  as
$h  \rightarrow \infty$, by Lemma \ref {le-3.1}, we obtain
\begin{eqnarray*}
    &&\int_\Omega  { | {{\nabla _{\mathcal{\widetilde{H}}_{\sigma}}}{v_h}(x)}  |dx}\\
    &&= \int_{\{ x \in \Omega :t < f(x) \le t + {1 \mathord{ /
    {\vphantom {1 h}}  . \kern-\nulldelimiterspace} h}\} } { | {{
    \nabla _{\mathcal{\widetilde{H}}_{\sigma}}}( {hf(x) - t} )
    }  |dx}  + \int_{\{ x \in
    \Omega :f(x) \ge t + {1 \mathord{ / {\vphantom {1 {h\} }}}  .
    \kern-\nulldelimiterspace} {h\} }}} { | {{\nabla _{{{\mathcal {\widetilde{H}}}}_
    \sigma}}1}  |dx} \\
    &&\leq \sqrt 2 h\int_{\{ x \in \Omega :t < f(x) \le t + {1 \mathord{
     / {\vphantom {1 h}}  . \kern-\nulldelimiterspace} h}\} }
    { | {\nabla f(x)}  |dx}  + \sqrt 2  | \sigma   |
    \int_{\{ x \in \Omega :t < f(x) \le t + {1 \mathord{ /{\vphantom
    {1 h}}  . \kern-\nulldelimiterspace} h}\} } {\frac{1}{{ |
    x  |}}dx}  \\
    &&\quad+| \sigma   |\sqrt 2 \int_{\{ x \in \Omega :f(x) \ge t + {1 \mathord{ / {
    \vphantom {1 {h\} }}}  . \kern-\nulldelimiterspace} {h\} }}}
    {\frac{{dx }}{{ | x  |}}}.
\end{eqnarray*}
Taking the limit $h \rightarrow \infty$ and using Theorem \ref
{th-3.3}, we obtain
\begin{equation}\label{eq-3.10}
    \begin{split}
         | {{\nabla _{\mathcal{\widetilde{H}}_{\sigma}}}{1_{{E_t}}}}  |(\Omega )
        %&\le \mathop {\lim \sup }\limits_{h \to \infty }  | {{\nabla _{
        %\mathcal{\widetilde{H}}_{\sigma}}}{v_h}}  |(\Omega )\\
        &\le \mathop {\lim \sup }\limits_{h \to \infty } \int_\Omega  { |
        {{\nabla _{\mathcal{\widetilde{H}}_{\sigma}}}{v_h}(x)}  |dx} \\
        &= \sqrt 2 m'(t) + \sqrt 2 \int_{ \{ {x\in\Omega :f(x) \ge t}
        \}} {\frac{{ | \sigma   |}}{{ | x  |}}dx} .
    \end{split}
\end{equation}
Integrating (\ref {eq-3.10}) reaches
\begin{align*}
    \int_{ - \infty }^{ + \infty } {{P_{\mathcal{\widetilde{H}}_{\sigma}}}({E_t},
    \Omega )} dt
    &\le \sqrt 2 \int_{ - \infty }^{ + \infty } {m'(t)} dt + \sqrt 2
    \int_{ - \infty }^{ + \infty } {\int_{ \{ {x \in \Omega :f(x) \ge t}
     \}} {\frac{{ | \sigma   |}}{{ | x  |}}dx} } dt\\
    &\le \sqrt 2 \int_\Omega  {( { | {\nabla f}  (x)| + \frac{{
     | \sigma   |}}{{ | x  |}} | {f(x)}  |}
    )dx} \\
    &\le \sqrt 2 \int_\Omega  { | {{\nabla _{\mathcal{\widetilde{H}}_{\sigma}}}f}(x)
     |dx}.
\end{align*}

Finally, by approximation and using the lower semicontinuity of the
 ${\mathcal{\widetilde{H}}_{\sigma}}$-perimeter, we conclude that  (\ref {coarea formula-1})
 holds true for all $f\in\mathcal {BV}_{\mathcal{\widetilde{H}}_{\sigma}}(\Omega)$
satisfying the condition (\ref{eq4}).
\end{proof}

In addition, we can develop the Sobolev's inequality and
the isoperimetric inequality for $\mathcal{\widetilde{H}}_{\sigma}$-BV
functions.
\begin{theorem}\label{thm2.7}
    \item{{\rm (i)}} (Sobolev  inequality)  Let $\Omega\subset{\mathbb{R} ^d}$
    be an open and bounded domain. For all $f \in {{\mathcal B}}
    {{{\mathcal V}}_{\mathcal{\widetilde{H}}_{\sigma}}}(\Omega)$ satisfying the
    condition (\ref{eq4}), then
\begin{equation}\label{eq12}
    { \| f  \|_{L^{{d}/{(d-1)}}(\Omega)}} \mathbin{\lower.3ex
    \hbox{$\buildrel<\over {\smash{\scriptstyle\sim}\vphantom{_x}}$}}
     | {{\nabla _{\mathcal{\widetilde{H}}_{\sigma}}}f}  |(\Omega).
\end{equation}

\item{{\rm (ii)}} (Isoperimetric inequality) Let $E$ be a bounded set
of finite ${\mathcal{\widetilde{H}}_{\sigma}}$-perimeter in $
\Omega$. Then
\begin{equation}\label{eq13}
       | E  |^{{1-1/d}}  \mathbin{\lower.3ex
    \hbox{$\buildrel<\over {\smash{\scriptstyle\sim}\vphantom{_x}}$}}
    {{P_{\mathcal{\widetilde{H}}_{\sigma}}}(E,\Omega )} .
\end{equation}

\item{{\rm (iii)}} The above two statements are equivalent.
\end{theorem}

\begin{proof}
(i) Let $ \{f_k\}_{k\in \mathbb N} \subseteq C^\infty_c(\Omega)\cap
\mathcal{BV}_{{\mathcal
    {\widetilde{H}}_\sigma}}(\Omega ) $  be a sequence such that
    $$\begin{cases}
        f_k \rightarrow f \quad\hbox{in}\quad L^1(\Omega), \\
        \int_{\Omega} |\nabla_{{\mathcal{\widetilde{H}}_{\sigma}}} f_k(x)|dx \rightarrow
        \big|  \nabla_{{\mathcal{\widetilde{H}}_{\sigma}}}  f\big|(\Omega).
    \end{cases}$$
Then by Fatou's lemma and the classical Gagliardo-Nirenberg-Sobolev inequality
(see \cite{EG}), we have
\[{ \| f  \|_{L^{{d}/{(d-1)}}(\Omega)}} \le \mathop {\lim\inf }
\limits_{k\to \infty } { \| {{f_k}} \|_{L^{{d}/{(d-1)}}(\Omega)}}
\mathbin {\lower.3ex\hbox{$\buildrel<\over
{\smash{\scriptstyle\sim}\vphantom {_x}}$}} \mathop {\lim
}\limits_{k \to \infty } { \| {\nabla f_k}   \|_{{L^1(\Omega)}}}
\mathbin{\lower.3ex\hbox{$\buildrel<\over {\smash{
\scriptstyle\sim}\vphantom {_x}}$}} \mathop {\lim }\limits_{k \to
\infty } { \| {{\nabla _{{{\mathcal{\widetilde{H}}_{\sigma}}}}}f_k}
\|_{{L^1} (\Omega)}} =  | {{\nabla
_{{\mathcal{\widetilde{H}}_{\sigma}}}}f} | (\Omega), \] where we
have used the relation between $ | {\nabla f(x)}
 |$ and $ | {{\nabla _{\mathcal{\widetilde{H}}_{\sigma}}}f(x)}
 |$ in Lemma \ref{le-3.1}.

(ii) We can show that (\ref{eq13}) is valid via letting $f=1_E$ in
(\ref{eq12}).

(iii) Obviously, (i)$\Rightarrow$(ii) has been proved. In what
follows, we prove (ii)$ \Rightarrow$(i). Assume that $0\le f\in
C^\infty_c(\Omega)$. By the coarea formula  in  Theorem \ref {coarea
formula} and (ii), we have
\begin{equation*}\label{equa2}
    \int_{\Omega} |\nabla_{{\mathcal{\widetilde{H}}_{\sigma}}} f(x)|dx
    \approx \int_{0}^\infty |
    \nabla_{{\mathcal{\widetilde{H}}_{\sigma}}} 1_{E_t}|(\Omega)\,dt
    \gtrsim\int_{0}^\infty |E_t|^{1-{1}/{d}}dt,
\end{equation*}
where $E_t=\big\{x\in \Omega:\ f(x)>t\big\}$. Let
    $$f_t=\min\{t,f\}\ \ \&\ \ \ \chi (t) = {\left( {\int_{\Omega} {f_t^{{d}/{(d-1)}}}(x) dx} \right)^{1 -{1}/{d}}}\quad \  \forall\ t\in \mathbb{R}.$$
It is easy to see that
    $$\lim_{t \rightarrow \infty}\chi(t)=\bigg(\int_{\Omega}
    |f(x)|^{{d}/{(d-1)}}dx\bigg)^{1-{1}/{d}}.$$

In addition, we can check that $\chi(t)$ is nondecreasing on
$(0,\infty)$ and for $h>0$, $$0\le \chi(t+h)-\chi(t)\le \bigg(
\int_{\Omega}
|f_{t+h}(x)-f_t(x)|^{{d}/{(d-1)}}dx\bigg)^{1-{1}/{d}} \le
h|E_t|^{1-{1}/{d}}.$$ Then $\chi(t)$ is locally a  Lipschitz function and
$\chi'(t)\le |E_t|^{1-{1}/{d}}$, a.e. $t \in (0,\infty )$.
Hence,
    \begin{equation*}
        \bigg(\int_{\Omega}|f(x)|^{{d}/{(d-1)}}dx\bigg)^
        {1-{1}/{d}}=\int^\infty_0 \chi'(t)dt\le\int^\infty_0  |E_t|^{1-{1}/{d}}dt
        \lesssim\int_{\mathbb{R}^d}|\nabla_{{\mathcal{\widetilde{H}}_{\sigma}}} f(x)|dx.
    \end{equation*}
For all $f \in {{\mathcal B}}{{{\mathcal
V}}_{\mathcal{\widetilde{H}}_{ \sigma}}}(\Omega)$ satisfying the
condition (\ref{eq4}), we conclude that (\ref{eq12}) is valid by
approximation and Theorem \ref{th-3.3}.
\end{proof}

\section{Subgraphs of ${\mathcal{\widetilde{H}}_{\sigma}}$-restricted BV functions}
         \label{sec-4}
\hspace{0.5cm} The aim of this section is to show that properties of
${\mathcal{\widetilde{H}}_{\sigma}}$-restricted BV functions can be
described equivalently   in terms of their subgraphs.

Let $\Omega \subseteq \mathbb{R}^{d}$ be an open set. The
${\mathcal{\widetilde{H}}_{\sigma}}$-restricted variation of $f \in
{L^1}(\Omega )$ is defined by \[ | {{\nabla^R
_{{\mathcal{\widetilde{H}}_{\sigma}}}}f}
 |(\Omega ) = \mathop {\sup } \limits_{\Phi  \in
{\widetilde{\mathcal F}_R(\Omega )}}  \Big\{ {\int_\Omega  {f(x)
\mathrm{div}{_{{\mathcal{\widetilde{H}}_{\sigma}}}}\Phi (x)dx} }
 \Big\},\] where $\widetilde{{\mathcal F}}_R(\Omega )$ denotes the
class of all functions
\[\Phi  = ( {{\varphi _1},{\varphi _2}, \ldots ,{\varphi _
{2d}}} ) \in C_c^1(\Omega, {\mathbb R^{2d}}) \] satisfying
\[{ \| \Phi    \|_\infty } = \mathop {\sup }\limits_
{x \in \Omega } \Big\{( {{{ | {{\varphi _1}(x)}  |}^2} + \ldots + {{
| {{\varphi _{2d}}(x)}  |}^2}} )^{{1}/{2}}\Big\} \le 1 \] and
\begin{equation}\label{restricted1}
   \int_{ {\Omega}} { \Big( {\sum\limits_{k = 1}^{d} {\sigma
   \frac{{{x_k}}}{{{{ | x  |}^2}}}( {{ {\varphi} _k}
   (x) - {{\varphi} _{k + d}}(x)} )} }  \Big)dx} =0.
\end{equation}

Define a new type BV space as: $${\mathcal{BV}^R
_{{\mathcal{\widetilde{H}}_{\sigma}} } ( \Omega
  ):=\{f\in L^1}(\Omega ):  | {{\nabla^R _
  {{\mathcal{\widetilde{H}}_{\sigma}}}}f}
 |(\Omega ) < \infty\} .$$ Similarly, it is easy to see that ${\mathcal{BV}
_{{\mathcal{\widetilde{H}}_{\sigma}} } ( \Omega
  )}\subseteq{\mathcal{BV}^R
_{{\mathcal{\widetilde{H}}_{\sigma}} } ( \Omega
  )}$, which also is a Banach space with the norm
\[ { \| f  \|_{\mathcal{BV}^R
        _{{\mathcal{\widetilde{H}}_{\sigma}} } ( \Omega
        )}} = { \| f  \|_{{L^{1}(\Omega)}}} +  |
{{\nabla ^R _ {{\mathcal{\widetilde{H}}_{\sigma}}}}f}  |(\Omega ).
\]
The space $ \mathcal{BV}^R _{{\mathcal{\widetilde{H}}_{\sigma}} } (
\Omega
  )$ enjoys similar  properties as $ \mathcal{BV} _{{\mathcal{\widetilde{H}}_{\sigma}} } (
\Omega
  )$, for example,  the lower semicontinuity, the structure theorem, the
approximation via $C^{\infty}_{c}$-functions, etc.

For $u \in \mathcal{BV}^R_ {\mathcal
{\widetilde{H}}_{\sigma}}(\Omega)$, the subgraph of $u$ is defined
as the measurable subset of $\Omega  \times \mathbb{R} $ given by
\begin{equation*}
{S_u}: = \{ (x,t) \in \Omega  \times \mathbb{R}:\ t < u(x)\} .
\end{equation*}

For the sake of simplicity, we introduce the family $ D: =
(\widetilde{A}_{1,a}, \ldots ,\widetilde{A}_{2d+1,a}) $ of linearly
independent vector fields in $ {\mathbb R^{2d + 1}}$ defined for $
(x,t) \in {\mathbb R^d} \times \mathbb R $ by
$$\left\{\begin{aligned}
\widetilde{A}_{i,a}(x,t)&:= ({A_{i,a}(x)},0) \in {\mathbb R^{d + 1}} \equiv
{\mathbb R^{d}} \times \mathbb R\ \ \ {\rm{for}}\  i=1,\ldots,d,\\
\widetilde{A}_{{d+i},a}(x,t)&:= ({A_{-i,a}}(x),0) \in {\mathbb R^{d
+ 1}}
\equiv {\mathbb R^{d}} \times \mathbb R\ \ {\rm{for}}\ i=1,\ldots,d,\\
\widetilde{A}_{2d+1,a}(x,t)&:=\frac{\partial}{\partial t}.
\end{aligned}\right.$$
Furthermore, we also need to define the so-called
${\mathcal{\widetilde{H}}_{\sigma}}$-restricted perimeter  in order
to achieve our aim.
\begin{definition}\label{def-restricte}
Let $\widetilde{\Omega}\subseteq \mathbb{R}^{d+1}$ be an open and
bounded domain. The ${\mathcal{\widetilde{H}}_{\sigma}}$-restricted
perimeter of $E \subseteq \widetilde{\Omega}$ can be defined as
\begin{equation*}
    {\widetilde P_{\mathcal{\widetilde{H}}_{\sigma}}}(E, \widetilde{\Omega})=
    \sup  \Bigg\{ {\int_E \sum^{2d+1}_{i=1} \widetilde{A}_{i,a}\widetilde{\varphi}_i
    (x,t)dxdt:\ \widetilde{\Phi } \in {{{
    \mathcal F}}_R}(\widetilde{\Omega},{\mathbb{R}^{2d+1}})} \Bigg\},
\end{equation*}
where ${{{\mathcal F}}_R}(\widetilde{\Omega},{\mathbb{R}^{2d+1}})$
denotes the class of all functions
\begin{equation*}
    \widetilde{\Phi}  = ( {{\widetilde{\varphi} _1},{\widetilde{\varphi }_2}, \ldots ,
    {\widetilde{\varphi} _{2d+1}}} ) \in
    C_c^1(\widetilde{\Omega},{\mathbb{R}^{2d+1}})
\end{equation*}
satisfying
\begin{equation*}
    {\big\| \widetilde{\Phi}  \big\|_\infty } = \mathop {\sup
    }\limits_{(x,t)
    \in {{\widetilde{\Omega}}}} \Big\{( {{{ | {{\widetilde{\varphi} _1}(x,t)}  |}
    ^2} +  \cdots  + {{ | {{\widetilde{\varphi} _{2d+1} }(x,t)}  |}^2}}
    )^{{1}/{2}}\Big\} \le 1
\end{equation*}
and
\begin{equation}\label{restricted}
   \int_{\widetilde{\Omega}} { \Big( {\sum\limits_{k = 1}^{d} {\sigma
   \frac{{{x_k}}}{{{{ | x  |}^2}}}( {{\widetilde{\varphi} _k}
   (x,t) - {\widetilde{\varphi} _{k + d}}(x,t)} )} }  \Big)dxdt} =0.
\end{equation}
\end{definition}

If $ \widetilde{\Omega} \subset {\mathbb R^{d+ 1}} $ is open and $u$
is a function of bounded variation on $ \widetilde{\Omega}$ with
respect to the family $D$, we write the $\mathbb{R} ^{2d+1} $-valued
distribution in $\widetilde{\Omega} $ as
\begin{equation*}
    Du: =
    ({D_{\widetilde{A}_{1,a}}}u, \ldots,{D_{\widetilde{A}_{2d+1,a}}}u).
\end{equation*}

The following theorem is the natural generalization of some related
results about   functions of bounded variation on the Euclidean
space and Carnot groups (see \cite {GMS} or \cite{DMV}). We denote
by $\pi :{\mathbb{R} ^{2d + 1}} \to {\mathbb{R} ^{2d}} $ the
canonical projection obeying $\pi (x,t) = x$ and ${\pi _\# } $
denotes the associated push-forward of measures.

\begin{theorem}\label{subgraph-1}
Let $\Omega  \subset {\mathbb R^d}$ be an open and bounded domain
and $u \in {L^1}(\Omega )$   satisfy the condition (\ref{eq4}). Then
$u$ belongs to $\mathcal{B} {\mathcal{V}^R_
{\mathcal{\widetilde{H}}_{\sigma}}}(\Omega )$  if and only if its
subgraph ${S_u}$ has finite
${\mathcal{\widetilde{H}}_\sigma}$-restricted perimeter in $\Omega
\times \mathbb{R} $, that is,
\begin{equation*}
{{\widetilde P}_{\mathcal{\widetilde{H}}_{\sigma}}}({S_u},\Omega
\times \mathbb{R}) < \infty.
\end{equation*}
Moreover, writing $D'{1_{{S_u}}}: = ( D_{\widetilde{A}_{1,a}}
{1_{{S_u}}}, \ldots ,D_{\widetilde{A}_{2d,a}} {1_{ {S_u}}})$, then
we have
    \item{{\rm(i)}} ${\pi _\# } D_{\widetilde{A}_{i,a}}{{1_{{S_u}}}}
      = D_{{A}_{i,a}}u$ and ${\pi _\# } D_{\widetilde{A}_{{d+i},a}}{{1_{{S_u}}}}   = D_{{A}_{-i,a}}u,
    i=1,\ldots,d;$
    \item{{\rm(ii)}} ${\pi _\# }{\frac{\partial}{\partial t}}{1_{{S_u}}} =  - {{{\mathcal
    L}}^d},$  where ${{{\mathcal L}}^d}$ is the Lebesgue measure on $\mathbb{R}^d$.
    \item{{\rm(iii)}} ${\pi _\# }\big|D_{\widetilde{A}_{i,a}}{{1_{{S_u}}}}\big|
    =\big| D_{{A}_{i,a}}u\big|$ and ${\pi _\# }\big|D_{\widetilde{A}_{{d+i},a}}{{1_{{S_u}}}}\big|
    =\big| D_{{A}_{-i,a}}u\big|, i=1,\ldots,d;$
    \item{{\rm(iv)}} ${\pi _\# }\big| {{\frac{\partial}{\partial t}}{1_{{S_u}}}} \big| =
          {{{\mathcal L}}^d}.$
\end{theorem}

\begin{proof}
Suppose first that ${{\widetilde
P}_{\mathcal{\widetilde{H}}_{\sigma}}}({S_u},\Omega  \times
\mathbb{R}) < \infty$. In this case,  the measures
$D_{\widetilde{A}_{i,a}}{1_{{S_u}}}$ can be extended as linear
functionals acting on continuous and bounded functions in $\Omega
\times \mathbb{R} $ by means of the Lebesgue theorem. We choose a
special sequence in $C_c^\infty \mathbb{(R} )$, denoted by $ \{
{{g_h}}  \}$,  such that
\[{g_h}(t) = \left \{ {\begin{array}{l}
        {1, \ \ \ | t  | \le h},\\
        {0, \ \ \ | t | \ge h + 1,}
\end{array}}  \right.\]
and
\[\int_\mathbb R {{g_h}(t)dt}  = 2h + 1. \]

Let $\Phi  =(\varphi_1,\ldots,\varphi_{2d}) \in C_c^1( \Omega,
{\mathbb{R} ^{2d}})$ with $ | \Phi  | \le 1$ and satisfy
(\ref{restricted1}).  By the dominated convergence theorem, we have
\begin{align*}
    \int_{\Omega  \times \mathbb R} {\Phi  (x)\cdot d(D'{1_{{S_u}}})(x,t)}
    &= \mathop {\lim }\limits_{h \to  + \infty } \int_{\Omega  \times
    \mathbb R} {{g_h}(t)\Phi  (x)\cdot d(D'{1_{{S_u}}})(x,t)}  \\
    &= \mathop {\lim }\limits_{h \to  + \infty } \int_{\Omega  \times
    \mathbb R} {{1_{{S_u}}}(x,t){g_h}(t)\mathrm{div}{_{\mathcal{\widetilde{H}}_\sigma}}
    \Phi  (x)dxdt} \\
    &= \mathop {\lim }\limits_{h \to  + \infty } \int_\Omega  {\Big( {
    \int_{ - \infty }^{u(x)} {{g_h}(t)dt} } \Big)\mathrm{div}{_{\mathcal{\widetilde{H}}_
    \sigma}}\Phi  (x)dx} .
\end{align*}

Via the definition of ${g_h}(t)$, for every $z \in \mathbb R$ and
every $h \in \mathbb N$, we  get
$$ \int_{ - \infty }^z {{g_h}(t)dt}
\le  | z  | + h + \frac{1}{2}$$
 and
$$\mathop {\lim }\limits_{h \to  + \infty } \Big( {\int_{ - \infty }^z {{g_h}
(t)dt}  - h - \frac{1}{2}} \Big) = z,$$
 where  we have used the fact
\[\int_\Omega  {{\rm{di}}{{\rm{v}}_ {\mathcal{\widetilde{H}}_{\sigma}}}}
\Phi  (x)dx = 0 . \]

Consequently, using the dominated convergence theorem again, we can deduce that
\begin{equation}\label{ eq-3.1}
    \begin{split}
    \int_{\Omega  \times \mathbb R} {\Phi (x)\cdot d(D'{1_{{S_u}}})(x,t)}
    &= \mathop {\lim }\limits_{h \to  + \infty } \int_\Omega  {\Big( {\int_{
    - \infty }^{u(x)} {{g_h}(t)dt}  - h - \frac{1}{2}}\Big)\mathrm{div}{_{{{
    \mathcal {\widetilde{H}}}}_\sigma}}\Phi  (x)dx} \\
    &= \int_\Omega  {u(x)\mathrm{div}{_{\mathcal{\widetilde{H}}_{\sigma}}}\Phi  (x)dx} \\
    &= \int_\Omega  {\Phi  (x)\cdot d({D _{\mathcal{\widetilde{H}}_{\sigma}}}u)} .
    \end{split}
\end{equation}
In particular, $u \in \mathcal{B} {\mathcal{V}^R
_{\mathcal{\widetilde{H}}_{\sigma}}}(\Omega )$ and for any open set
$A \subseteq \Omega $, we have
\begin{eqnarray}\label{S-2}
&&  | {{D_{\mathcal{\widetilde{H}}_{\sigma}}}u}   |(A) \le   |
{D'{1_{{S_u} }}}   |(A \times \mathbb{R} ).
\end{eqnarray}

Before proving the reverse implication, we firstly  consider two
facts.   For any $\psi  \in C_c^1(\Omega)$ one has
\begin{equation}\label{ eq-3.2}
    \begin{split}
    \int_{\Omega  \times \mathbb R} {\psi  (x)d({D_{_{\widetilde{A}_{2d+1,a}}}}{1_{{S_u}}})
    (x,t)}
    &=\int_{\Omega  \times \mathbb R} {\psi  (x)d({\frac{\partial}{\partial t}}{1_{{S_u}}
    })(x,t)}\\
    &= \mathop {\lim }\limits_{h \to  + \infty } \int_{\Omega  \times
    \mathbb R} {{g_h}(t)\psi  (x)d({\frac{\partial}{\partial t}}{1_{{S_u}}})(x,t)} \\
    &=  - \mathop {\lim }\limits_{h \to  + \infty } \int_{\Omega  \times
    \mathbb R} {{1_{{S_u}}}(x,t){{g'_h}}(t)\psi  (x)dxdt} \\
    &=  - \mathop {\lim }\limits_{h \to  + \infty } \int_\Omega  {\Big( {
    \int_{ - \infty }^{u(x)} {{g'_h}(t)dt} } \Big)\psi  (x)d x} \\
    &=  - \mathop {\lim }\limits_{h \to  + \infty } \int_\Omega  {{g_h}(u(x))
    \psi  (x)dx} \\
    &=  - \int_\Omega  {\psi  (x)dx} .
    \end{split}
\end{equation}
Moreover, for any open set $A \subset \Omega $,
\begin{equation}\label{S-3}
   |A|\le  | {{D_{_{\widetilde{A}_{2d+1,a}}}}{1_{{S_u}}}}  |
    (A \times \mathbb{R} ).
\end{equation}

Furthermore, if $\Phi   \in C_c^1(\Omega ,{\mathbb{R}^{2d}})$
  satisfies
$$\int_{ {\Omega}} { \Big( {\sum\limits_{k = 1}^{d} {\sigma
   \frac{{{x_k}}}{{{{ | x  |}^2}}}( {{ {\varphi} _k}
   (x) - { {\varphi} _{k + d}}(x)} )} }  \Big)dx} =0,$$  by (\ref{ eq-3.1})
and (\ref{ eq-3.2}),
we can get
\[\int_{\Omega  \times \mathbb R} {\Phi  (x)\cdot d(D{1_{{S_u}}})(x,t)}  =
\int_\Omega  {\Phi  (x)\cdot d({D
_{\mathcal{\widetilde{H}}_{\sigma}}}u, - {\mathcal {L} ^d})(x)} ,\]
which derives for any open set $A \subset \Omega $,
\begin{equation}\label{S-4}
    | {({D_{\mathcal{\widetilde{H}}_{\sigma}}}u, - {\mathcal{L} ^d})}
    |(A) \le  | {D{1_{{S_u}}}}  |(A \times \mathbb{R} ).
\end{equation}

Suppose now that $u \in \mathcal{B} {\mathcal{V}^R
_{\mathcal{\widetilde{H}}_{\sigma}}} (\Omega )$  satisfies the
condition (\ref{eq4}). Let $A \subset \Omega$ be open. Similarly to
Theorem \ref{th-3.3}, we can choose a sequence of smooth functions $
\{ {{u_k}}  \}$ in $C_c^\infty (A) \cap {{\mathcal B}}{{{\mathcal
V}}^R_{{{\mathcal {\widetilde{H}}}}_\sigma}}(A)$ such that ${u_k}
\to u$ in ${L^1}$ and
\[\int_A { | {{\nabla^R _{{{ \mathcal {\widetilde{H}}}}_\sigma}}{u_k}} (x) |dx}
\to  | {{\nabla^R _{\mathcal{\widetilde{H}}_{\sigma} }}u}  |(A)\] as
$k \to \infty $. In the classical case (cf. \cite[Theorem 1]{GMS}),
we observe that for any $\widetilde{\Phi}   \in C_c^1 (A \times
\mathbb R)$,   \[\int_{A \times \mathbb R} {\widetilde{\Phi} (x,t)
d(
\partial_ {x_i}{1_{{S_{{u_k}}}}}) = \int_A {\widetilde{\Phi} (
{x,{u_k}(x)} ) \frac{\partial}{\partial x_i}{u_k}(x)dx} },\  i = 1,
\ldots ,d.\]

For convenience, we write
 $$\bar \nabla {1_{{S_{{u_k}}}}} = \Big(
\frac{\partial}{\partial x_1} {1_{{S_{{u_k}}}}}, \ldots
,\frac{\partial}{\partial x_d}{1_{{S_{{u_k}}}}},
\frac{\partial}{\partial x_1} {1_{{S_{{u_k}}}}}, \ldots ,
\frac{\partial}{\partial x_d}{1_{{S_{{u_k}}}}},
\partial_{t} {1_{{S_{{u_k}}}}}\Big) $$
and
 $$\bar \nabla
'{1_{{S_{{u_k}}}}} = \Big( \frac{\partial}{\partial x_1}{1_
{{S_{{u_k}}}}}, \ldots , \frac{\partial}{\partial
x_d}{1_{{S_{{u_k}}}}},\frac{\partial}{\partial {x_1}}{1_{
{S_{{u_k}}}}}, \ldots , \frac{\partial}{\partial
x_d}{1_{{S_{{u_k}}}}}\Big).$$
Then for any $\widetilde{\Phi} \in C_c^1
(A \times \mathbb R)$ with $\big| \widetilde{\Phi} \big| \le 1$ and
satisfying (\ref{restricted}), we have
\begin{equation}\label{eq-3.6}
\begin{split}
&\int_{A \times \mathbb R} {\widetilde{\Phi} (x,t)\cdot d{(\bar \nabla '{1_{{S_{{u_k}}}}})}}\\
&= \int_{A \times \mathbb R} {\widetilde{\Phi} (x,t)\cdot d\Big( \frac{\partial}{\partial x_1}{1_{{S_
{{u_k}}}}}, \ldots , \frac{\partial}{\partial x_d}{1_{{S_{{u_k}}}}}, \frac{\partial}{\partial x_1}
{1_{{S_{{u_k}}}}}, \ldots , \frac{\partial}{\partial x_d}{1_{{S_{{u_k}}}}}\Big)}  \\
    &=\int_A {\widetilde{\Phi} ( {x,{u_k}(x)} )\cdot\Big( \frac{\partial}{\partial x_1}{u_k}(x),
    \ldots , \frac{\partial}{\partial x_d}{u_k}(x), \frac{\partial}{\partial x_1}{u_k}(x), \ldots ,
    \frac{\partial}{\partial x_d}{u_k}(x)\Big)dx}  \\
    &+ \int_A {\widetilde{\Phi} ( {x,{u_k}(x)} )\cdot\Big(\sigma \frac{{{x_1}}}
    {{{{ | x  |}^2}}}{u_k}(x), \ldots ,
    \sigma \frac{{{x_d}}}{{{{ | x  |}^2}}}{u_k}(x), - \sigma \frac{{{x_1}}}
    {{{{ | x  |}^2}}}{u_k}(x), \ldots ,
    - \sigma \frac{{{x_d}}}{{{{ | x  |}^2}}}{u_k}(x)\Big)dx}
    \\
    &= \int_A {{u_k}(x)\mathrm{div}{_{\mathcal{\widetilde{H}}_{\sigma}}}\widetilde{\Phi} (
    {x,{u_k}(x)} )dx} \\
    &=\int_A {\widetilde{\Phi} ( {x,{u_k}(x)} )d({D_{{{\mathcal
    {\widetilde{H}}}}_\sigma}}{u_k})},
    \end{split}
\end{equation}
where we have used the fact that $x \to  \int_{ - \infty }^{{u_k}} {\Phi
(x,t)dt} $ is in $C_c^1(A) $. In a similar way,
\begin{equation}\label{eq-3.7}
    \begin{split}
    \int_{A \times \mathbb R} {\widetilde{\Phi}(x,t) d\Big({D_{_{\widetilde{A}_{2d+1,a}}}}{1_{{S_u}}}\Big)}
    &=\int_{A \times \mathbb R} {\widetilde{\Phi}(x,t) d({\frac{\partial}{\partial t}}{1_{{S_{{u_k}}}}})}\\
    &= -\int_A {\Big( {\int_{ - \infty }^{{u_k}(x)} {{\frac{\partial}{\partial t}}\widetilde{\Phi}
    (x,t)dt} } \Big)dx} \\
    &=  - \int_A {\widetilde{\Phi} (x,{u_k}(x))dx} .
    \end{split}
\end{equation}

Formulas (\ref {eq-3.6}) and (\ref {eq-3.7}) imply that for any
$\widetilde{\Phi}  \in C_c^1(A \times \mathbb R,{\mathbb R^{2d +
1}})$ satisfying (\ref{restricted}),
\[\int_{A \times \mathbb R} {\widetilde{\Phi} (x,t) \cdot d({\bar \nabla }
{1_{{S_{{u_k}}}}})}  = \int_A {\widetilde{\Phi} (x,{u_k}(x))\cdot
d({D _{ \mathcal{\widetilde{H}}_{\sigma}}}{u_k},{{-{\mathcal{L}
}}^d})(x)} .\] Since ${1_{{S_{{u_k}}}}} \to {1_{{S_u}}}$ in ${L^1}(A
\times \mathbb R)$, we   get
\begin{equation}\label{eq-3.8}
    \begin{split}
     | {D{1_{{S_u}}}}  |(A \times \mathbb R)
    &= | {\bar \nabla {1_{{S_u}}}}  |(A \times \mathbb R)\\
    &\le \mathop {\lim \inf }\limits_{k \to  + \infty }  | {\bar
    \nabla {1_{{S_{{u_k}}}}}}  |(A \times \mathbb R)\\
    &\le \mathop {\lim \inf }\limits_{k \to  + \infty }  | {({D _
    {\mathcal{\widetilde{H}}_{\sigma}}}{u_k}, - {{{\mathcal L}}^d})}  |(A)\\
    &=  | {({D _{\mathcal{\widetilde{H}}_{\sigma}}}{u}, - {{{\mathcal L}}^d}
    )}  |(A) <  + \infty,
    \end{split}
\end{equation}
which indicates that ${{\widetilde
P}_{\mathcal{\widetilde{H}}_{\sigma}}}({S_u},\Omega  \times
\mathbb{R}) < \infty$. Similarly, using the lower semicontinuity, we
obtain
\begin{equation}\label{eq-3.9}
    \begin{split}
     | {D'{1_{{S_u}}}}  |(A \times \mathbb R) \le  | {{D_
    {\mathcal{\widetilde{H}}_{\sigma}}}u}  |(A), \\
     | {{D_{_{\widetilde{A}_{2d+1,a}}}}{1_{{S_u}}}}  |(A \times \mathbb R) \le |A| <  + \infty .
    \end{split}
\end{equation}

Eventually, statements (i) and (ii) follow from (\ref { eq-3.1}) and
(\ref{ eq-3.2}), while statements (iii) and (vi) are consequences of
formulas (\ref{S-2}), (\ref{S-3}), (\ref{S-4}), (\ref{eq-3.8}) and
(\ref{eq-3.9}).
\end{proof}

Let $u =(u_1,\ldots, u_m)\in {{\mathcal B}}{{{\mathcal
V}}^R_{\mathcal{\widetilde{H}}_{\sigma}}}(\Omega,{ \mathbb R^m})$,
that is, $ u_i\in {{\mathcal B}}{{{\mathcal
V}}^R_{\mathcal{\widetilde{H}}_{\sigma}}}(\Omega)$ for
$i=1,\ldots,m$. By the Lebesgue decomposition theorem for measures,
we can decompose its distributional   derivatives as $${D
_{\mathcal{\widetilde{H}}_{\sigma}}}u =D
_{\mathcal{\widetilde{H}}_{\sigma}}^Au + D
_{\mathcal{\widetilde{H}}_{\sigma}}^Su,$$ where $D
_{\mathcal{\widetilde{H}}_{\sigma}}^Au$ is absolutely continuous
with respect to the Lebesgue measure ${{{\mathcal L}}^d}$ and $D
_{\mathcal{\widetilde{H}}_{\sigma}}^Su$ is singular with respect to
${{{\mathcal L}}^d}$. Furthermore, write
$D_{\mathcal{\widetilde{H}}_{\sigma}}^Au = Mu{{{\mathcal L}}^d}$,
where $Mu \in L_{loc}^1(\Omega ,{\mathbb R^{2d}})$ is the
approximate differential of $u$.

In this case, the Radon-Nikodym derivative $\frac{{D
_{\mathcal{\widetilde{H}}_{\sigma}}^Su}}{{ | {D
_{\mathcal{\widetilde{H}}_{\sigma}}^Su}  |}}$ of ${D
_{\mathcal{\widetilde{H}}_{\sigma}}^Su}$ with respect to its total
variation ${ | {D_{\mathcal{\widetilde{H}}_{\sigma}}^Su}  |}$ is a
${ | { D _{\mathcal{\widetilde{H}}_{\sigma}}^Su}
 |}$-measurable map from $\Omega $ to ${\mathbb R^{{2d} \times
m}}$. We can define the  normal to $x$ as $${v_S}(x): =\frac{{D
_{\mathcal{\widetilde{H}}_{\sigma}}^Su}}{{ | { D
_{\mathcal{\widetilde{H}}_{\sigma}}^Su}  |}}\in {\mathbb R^{2d}}.$$
The normal ${v_S}(x) = ( {{{( {{v_S}(x)} )}_1}, \ldots, {{(
{{v_S}(x)} )}_{2d}}} )$ is defined up to sign and it can be
canonically identified with a   vector at $x$ by $${v_S}(x) = {(
{{v_S}(x)} )_1}{D_{{A_{1,a}}}}(x) + \ldots + {( {{v_S}(x)}
)_{2d}}{D_{{A_{-d,a}}}}(x).$$

We can also consider the polar decomposition ${D
_{\mathcal{\widetilde{H}}_{\sigma}}}u = {\sigma _u}\big| {{D
_{\mathcal{\widetilde{H}}_{\sigma}}}u} \big|$, where ${\sigma
_u}:\Omega \to {\mathbb S^{2d - 1}}$ is a $ | {{D
_{\mathcal{\widetilde{H}}_{\sigma}}}u}  |$-measurable function. If
$u = {1_S}$ is the characteristic function of a set $S \subset
\Omega \times \mathbb R$ of finite
$\widetilde{\mathcal H}_\sigma$-restricted perimeter in $\Omega \times
\mathbb R$, we write $D{1_S} = {v_S} \big| D 1_S \big|$
%${\nabla _{\widetilde P}}{1_E} = {v_E} | {{\nabla _{\widetilde P}}{1_E}}  |$
for some Borel function ${v_S} = ( {{{({v_S})}_1}, \ldots
,{{({v_S})}_{2d+1}}} )$ called  inner normal to $S$.

Via Theorem \ref{subgraph-1}, we  can deduce the  following result.
\begin{theorem}\label{subgraph-2}
Let $\Omega  \subset {\mathbb R^d}$ be a bounded open  domain. If $u
\in \mathcal B{\mathcal V^R_{\mathcal{\widetilde{H}}_{\sigma}}}
(\Omega )$    satisfies the condition (\ref{eq4}), define
\[E: = \Big\{ (x,t) \in \Omega  \times \mathbb R:\ {(v_{S_u})_{2d + 1}}(x,t) = 0\Big\} \]
and
\[T: = \Big\{ (x,t) \in \Omega  \times \mathbb R:\ {(v_{S_u})_{2d + 1}}(x,t) \ne 0\Big\} .\]
Then, the following identities are valid:
\begin{equation}\label{eq-3.1}
v_{S_u}(x,t) = ({\sigma _u}(x),0){\rm{~for~}}{{\widetilde
P}_{\mathcal{\widetilde{H}}_{\sigma}}} ({S_u},\Omega\times \mathbb
R)-a.e.\ (x,t) \in E;
\end{equation}
\begin{equation}\label{eq-3.2}
v_{S_u}(x,t) = \frac{{(Mu(x), - 1)}}{{\sqrt {1 + {{ | {Mu(x)}
 |}^2}} }}{\rm{~for~}}{{\widetilde
P}_{\mathcal{\widetilde{H}}_{\sigma}}}({S_u}, \Omega\times \mathbb
R)-a.e. {\rm{~}}(x,t) \in T;\\
\end{equation}
\begin{equation}\label{eq-3.3}
 {\pi _\# }(D{1_{{S_u}}} \llcorner E) = (D _{\mathcal{\widetilde{H}}_{\sigma}}^Su,0);
\end{equation}
\begin{equation}\label{eq-3.4}
{\pi _\# }(D{1_{{S_u}}} \llcorner T) = (D
_{\mathcal{\widetilde{H}}_{\sigma}}^Su, - {{{\mathcal L}}^d}).
\end{equation}
\end{theorem}

\begin{proof}
From the above Theorem \ref{subgraph-1} and similarly to
\cite[Theorem 2.28]{AFP}, we can decompose the perimeter
${{\widetilde P}_{\mathcal{\widetilde{H}}_{\sigma}}}({S_u},
\Omega\times \mathbb R)$ into $ | {({D
_{\mathcal{\widetilde{H}}_{\sigma}}}u, - {{{\mathcal L}}^d})}  |$:
for every $x \in \Omega $, there exists a probability measure ${\mu
_x}$ on $\mathbb R$ such that for every Borel function $g \in
{L^1}(\Omega  \times \mathbb R, {{\widetilde
P}_{\mathcal{\widetilde{H}}_{\sigma}}}({S_u},\Omega\times \mathbb
R))$,
\begin{equation*}
\int_{\Omega  \times \mathbb R} {g(x,t)d{{\widetilde
P}_{\mathcal{\widetilde{H}}_{\sigma}}} ({S_u},\Omega\times \mathbb
R)(x,t)}  = \int_\Omega {\Big( {\int_\mathbb R {g(x,t)d{\mu _x}(t)} }
\Big)d} | {({D_{\mathcal{\widetilde{H}}_{\sigma}}}u, - {{{\cal L}}^d})}
 |(x).
\end{equation*}
Therefore, for any Borel function $\Phi :\Omega  \to \mathbb R$, we can get
\begin{equation} \label{eq3.1}
    \begin{aligned}
    \int_\Omega  {\Phi (x)d({D _{\mathcal{\widetilde{H}}_{\sigma}}}u, - {{{
    \mathcal L}}^d})} (x)
    &= \int_\Omega  {\Phi (x)d{\pi _\# }({v_{{S_u}}}{{\widetilde P}_{\mathcal
    {\widetilde{H}}_{\sigma}}}({S_u}, \Omega\times \mathbb R))} (x)\\
    &= \int_{\Omega  \times \mathbb R} {\Phi (x){v_{{S_u}}}(x,t)d({{
    \widetilde P}_{\mathcal{\widetilde{H}}_{\sigma}}}({S_u},\Omega\times \mathbb R))} (x,t)\\
    &= \int_\Omega  {\Phi (x)\Big( {\int_\mathbb R {{v_{{S_u}}}(x,t)
    d{\mu _x}(t)} } \Big)d | {({D _{\mathcal{\widetilde{H}}_{\sigma}}}u, - {{{
    \cal L}}^d})}  |(x)}.
    \end{aligned}
\end{equation}
Since $D _{\mathcal{\widetilde{H}}_{\sigma}}^Au$ and $D
_{\mathcal{\widetilde{H}}_{\sigma}}^Su$ are mutually singular, we
have
$$ | {({D_{\mathcal{\widetilde{H}}_{\sigma}}}u, - {{{\mathcal L}}^d})}
 |= | {(D_{\mathcal{\widetilde{H}}_{\sigma}} ^Au, - {{{\mathcal
L}}^d})}  | +  | {(D_{{{\mathcal {\widetilde{H}}}}_\sigma}^ Su,0)}
 |=\sqrt {1 + {{ | {Mu}  |}^2}} {{{\mathcal L}}^
d}+ | {D _{\mathcal{\widetilde{H}}_{\sigma}}^Su}  |,
$$ and   (\ref{eq3.1}) gives
\begin{equation*} \label{eq3.2}
    \int_\Omega  {\Phi (x)d( {(Mu, - 1){{{\mathcal L}}^d} +
    ({\sigma _u},0) | {D _{\mathcal{\widetilde{H}}_{\sigma}}^Su}  |} )}
    {\rm{~~~~~~~~~~~~~~~~~~~~~~~~~~~~~~~}}
\end{equation*}
\begin{equation}\label{eq3.2}
    = \int_\Omega  {\Phi (x)\Big( {\int_\mathbb R {v_{S_u}(x,t)d{\mu _x}
    (t)} } \Big)d( {\sqrt {1 + {{ | {Mu}  |}^2}} {{{\mathcal L}}
    ^d}+ | {D_{\mathcal{\widetilde{H}}_{\sigma}}^Su}  |} )} (x).
\end{equation}

Let $I$ denote the subset of $\Omega$ such that its Lebesgue measure
$|I| = 0$ and $ | {D _{\mathcal{\widetilde{H}}_{\sigma}}^Su}
 |(\Omega \backslash I) = 0$. Considering Borel test functions
$\varphi$ such that $\varphi=0$ in $\Omega \backslash I$, we deduce
that for $ | {D _{\mathcal{\widetilde{H}}_{\sigma}}^Su}
 |$-a.e. $x \in I$ one has
\[({\sigma _u}(x),0) = \int_\mathbb R {v_{S_u}(x,t)d{\mu _x}(t)} .
\] Taking the scalar product with $({\sigma _u} (x),0)$ on both
sides, we can get
\begin{equation*}
\left \langle {({\sigma _u}(x),0),\int_\mathbb R {v_{S_u}(x,t)d{\mu
_x}(t) } }  \right\rangle  = 1.
\end{equation*}

Since ${{\mu _x}(\mathbb R) = 1}$ and (for $ | {({D _{{{\mathcal
{\widetilde{H}}} }_\sigma}}u, - {{{\mathcal L}}^d})}  |$-a.e. $x \in
\Omega $) $ | {v_{S_u} (x,t)}  | = 1$ for ${\mu _x}$-a.e. $t$, we
conclude that
\begin{equation*}
    v_{S_u}(x,t) = ({\sigma _u}(x),0){\rm{~for}}\  | {D _{\mathcal{
    \widetilde{H}}_{\sigma}}^Su} |{\rm{-a.e.~}}x \in I {\rm{~and~}}{\mu
_x}{\rm{-a.e.~}}t\in
    \mathbb R,
\end{equation*}
i.e.,
\begin{equation}\label{eq3.4}
   v_{S_u}(x,t) = ({\sigma _u}(x),0) {\rm{~for~}} {{\widetilde P}_{\mathcal
   {\widetilde{H}}_{\sigma}}}({S_u}, \Omega\times \mathbb R){\rm{-a.e.~}}(x,t)
   \in I\times \mathbb R,
\end{equation}
which implies that ${{\widetilde
P}_{\mathcal{\widetilde{H}}_{\sigma}}}({S_u}, \Omega\times \mathbb
R)$-a.e. $(x,t) \in I\times \mathbb R$ belongs to $E$ and that
${{\widetilde
P}_{\mathcal{\widetilde{H}}_{\sigma}}}({S_u},\Omega\times \mathbb
R)$-a.e. $(x,t) \in T$ belongs to $(\Omega \backslash I) \times
\mathbb R$.

Using  (\ref{eq3.2}) again and letting $\Phi =0$ on $I$, we obtain
\begin{equation*}
\begin{split}
    &\int_\Omega  {\Phi (x)\frac{{(Mu(x), - 1)}}{{\sqrt {1 + {{ |
    {Mu(x)}  |}^2}} }}\sqrt {1 + {{ | {Mu(x)}  |}^2}} dx} \\
    &=\int_\Omega  {\Phi (x)\Big( {\int_\mathbb R {v_{S_u}(x,t)d{\mu _x}
    (t)} } \Big)\sqrt {1 + {{ | {Mu(x)}  |}^2}} dx} .
\end{split}
\end{equation*}
Then, for  a.e. $x \in \Omega \backslash I$, we have
\[\int_\mathbb R {v_{S_u}(x,t)d{\mu _x}(t)}  = \frac{{(Mu(x), - 1)}}{{\sqrt
{1 + {{ | {Mu(x)}  |}^2}} }}. \] Consequently, for  a.e. $x \in
\Omega \backslash I $ and ${\mu _x}$-a.e. $t\in\mathbb R$, we can
deduce that
\begin{equation*}
    v_{S_u}(x,t) = \frac{{(Mu(x), - 1)}}{{\sqrt {1 + {{ | {Mu(x)}  |}
    ^2}} }},
\end{equation*}
or equivalently, for ${{\widetilde
P}_{\mathcal{\widetilde{H}}_{\sigma}}}({S_u},\Omega\times \mathbb
R)$-a.e. $(x,t) \in (\Omega \backslash I)\times \mathbb R, $
\begin{equation*}\label{eq-3.5}
    v_{S_u}(x,t) = \frac{{(Mu(x), - 1)}}{{\sqrt {1 + {{ | {Mu(x)}  |}
    ^2}} }}.
\end{equation*}
Similarly, it implies that ${{\widetilde
P}_{\mathcal{\widetilde{H}}_{\sigma}}}({S_u},\Omega\times \mathbb
R)$-a.e. $(x,t) \in (\Omega \backslash I)\times \mathbb R$ belongs
to $T$ and   ${{\widetilde
P}_{\mathcal{\widetilde{H}}_{\sigma}}}({S_u},\Omega\times \mathbb
R)$-a.e. $(x,t) \in E$ belongs to $I \times \mathbb R$.

Since $E$ and $T$ are disjoint, the formulas (\ref{eq-3.1}) and
(\ref{eq-3.2}) can be obtained. Now   (\ref{eq-3.3}) can be easily
deduced due to
\begin{equation*}
    \begin{aligned}
  {\pi _\# }(D{1_{{S_u}}}\llcorner E)
  &= {\pi _\# }( {v_{S_u}{{\widetilde P}_{\mathcal{\widetilde{H}}_{\sigma}}}
  ({S_u},\Omega\times \mathbb R)\llcorner(I \times \mathbb R)}  ) \\
  &=({\sigma _u}(x),0) | {({D _{\mathcal{\widetilde{H}}_{\sigma}}}u, - {{{
  \mathcal L}}^d})}  |\llcorner I \\
  &= (D_{\mathcal{\widetilde{H}}_{\sigma}}^Su,0). \\
    \end{aligned}
\end{equation*}

The last formula (\ref{eq-3.4}) can be obtained by the formula (\ref
{eq3.4}) and similarly,
\begin{equation*}
    \begin{aligned}
    {\pi _\# }(D{1_{{S_u}}}\llcorner T)
    &= {\pi _\# }( {v_{S_u}{{\widetilde P}_{\mathcal{\widetilde{H}}_{\sigma}}}({S_u},
    \Omega\times \mathbb R)
    \llcorner ( {(\Omega \backslash I) \times \mathbb R} )} )\\
    &=\frac{{(Mu, - 1)}}{{\sqrt {1 + {{ | {Mu(x)}  |}^2}} }} | {({
    D_{\mathcal{\widetilde{H}}_{\sigma}}}u, - {{{\mathcal L}}^d})}  |\llcorner
    (\Omega \backslash I)\\
    &= (Mu, - 1){{{\mathcal L}}^d}.
    \end{aligned}
\end{equation*}This completes the proof of this theorem.
\end{proof}

\section{Rank-one theorem for ${\mathcal{\widetilde{H}}_{\sigma}}$-restricted BV functions}\label{sec-5}
\hspace{0.6cm} In this section, we prove the rank-one theorem for
$\mathcal{\widetilde{H}}_{\sigma}$-restricted BV functions in
Euclidean spaces by using Theorem \ref {subgraph-1} in Section 4 and
Lemma \ref {sub-1} below which is a key tool.

Let ${\mathcal H^{d-1}}$ be the standard ($d-1$)-dimensional Hausdorff
measure. A set $E \subset {\mathbb R^{d}}$ is rectifiable if
${\mathcal H^{d-1}}(E) < \infty $ and there exists a (finite or
countable) family  of $C^1$ hypersurfaces in
$\mathbb R^{d}$, denoted by $\{{\Sigma _i}\}_{i\in\mathbb N}$, such that
\[{\mathcal H^{d-1}}(E\backslash \mathop  \cup \limits_{i\in\mathbb N} {\Sigma _i}) = 0.\]
We define the  normal ${v_E}$ to $E$ as
\begin{equation*}
{v_E}(x): = {v_{{\Sigma_i}}}(x)\ \ \ \ \mathrm{if}\  x \in E \cap
{\Sigma _i}\backslash { \cup _{j < i}}{\Sigma _j}.
\end{equation*}
Note that the normal ${v_E}$ is well-defined (up to sign) ${\mathcal
H^{d-1}}$-a.e. on $E$,  since the set of points where two $C^1$
hypersurfaces intersect transversally is ${\mathcal
H^{d-1}}$-negligible.

\begin{definition}
Let $ E $ be of  finite ${{\widetilde
P}_{\mathcal{\widetilde{H}}_{\sigma}}}$-perimeter. The
$\mathcal{\widetilde{H}}_{\sigma}$-reduced boundary of $E$, denoted
by ${\partial _{\mathcal{\widetilde{H}}_{\sigma}}}E$, consists of
all points $x \in {\mathbb R^{d+1}}$ for which the following
statements hold:
 \item{{\rm(i)}}\ $ | {D{1_E}}  |( {B(x,r)} ) > 0$
 for all $r>0$,

\item{{\rm(ii)}}\ if
$${n_r}(x,E) =  -\frac{D1_E(B(x,r))}{|D1_E|(B(x,r)) },$$
then the limit $n(x,E) := \mathop {\lim }\limits_{r \to 0}
{n_r}(x,E)$ exists with $ | {n(x,E)}  | = 1$.
\end{definition}

\begin{remark}\label{re-1}
Similarly to \cite[Section 5.5]{Ziemer}, we deduce that $ | {D{1_E}}
|({\mathbb R^{d + 1}} - {\partial _{{{{\widetilde {\cal H}}}_\sigma
}}}E) = 0$ is true. Consequently, ${\partial
_{\mathcal{\widetilde{H}}_{\sigma}}}E$ is $ | {D{1_E}}
 |$-measurable and $ | {D{1_E}}  | =  | {D{1_E}}
 |\llcorner{\partial _{{{{\widetilde {\cal H}}}_\sigma }}}E$. Then
by the rectifiability theorem for measures \cite [Section 2.1.4,
Theorem 2]{GMS}, the  rectifiability of ${\partial
_{\mathcal{\widetilde{H}}_{\sigma}}}E$ can be obtained.
Correspondingly, it follows that  $ | {D{1_E}}  | = {\mathcal
H^{d}}\llcorner{\partial _{\mathcal{\widetilde{H}}_{\sigma}}}E$ if
$E \subset \mathbb R^{d+1}$.
\end{remark}

Via the lemma in   \cite[page 3256]{MV} or \cite[Lemma 3.2]{DMV},
the following result can be derived from the coarea formula in
Theorem \ref{coarea formula}.
\begin{lemma}\label{sub-1}
Let   ${\Sigma _1}$, ${\Sigma _2}$ be $C^1$ hypersurfaces in
${\mathbb R^{d + 1}}$ with unit normals ${v_{{\Sigma _1}}}$ and
${v_{{\Sigma _2}}}$. Then the set
\begin{equation*}
    T: = \Big\{ p \in {\Sigma _1}:\exists q \in {\Sigma _2} \cap
    {\pi ^{ - 1}}(\pi (p)){\rm{~}}\ with\ {\rm{~}}{({v_{{\Sigma _1
    }}}(p))_{2d + 1}} = {({v_{{\Sigma _2}}}(q))_{2d + 1}} = 0{\rm{~}}\
    and\
\end{equation*}
\begin{equation*}
    {\rm{~~~~~~~~~~~~~~~~~~~~~~~~~~~~~~~~~~~~~~~~~~~~~~~~~~~~~~~~~~~
    ~~~~~~~~~~~~~~~~~~~~~}}{v_{{\Sigma _1}}}(p) \ne  \pm {v_{{\Sigma _
    2}}}(q)\Big\}
\end{equation*}
is ${\mathcal H^{d}}$-negligible.
\end{lemma}

\begin{theorem}\label{ Rand-one }(Rank-one theorem)
Let $\Omega  \subset {\mathbb R^d}$ be an open and bounded domain.
Assume $u\in {{\mathcal B}}{{{\mathcal
V}}^R_{\mathcal{\widetilde{H}}_{\sigma}}}(\Omega, {\mathbb R^m})$ is
a function with   bounded $ \mathcal{\widetilde{H}}_{\sigma}
$-variation  and satisfies the condition (\ref{eq4}). Let $D
_{\mathcal{\widetilde{H}}_{\sigma}}^Su$ be the singular part of ${D
_{\mathcal{\widetilde{H}}_{\sigma}}}u$ with respect to the Lebesgue
measure ${\mathcal{L} ^d}$. Then $D
_{\mathcal{\widetilde{H}}_{\sigma}}^Su$ is a rank-one measure, i.e.,
the (matrix-valued) function $\frac{{D
_{\mathcal{\widetilde{H}}_{\sigma}}^Su}}{{ | {D
_{\mathcal{\widetilde{H}}_{\sigma}}^Su}  |}}(x)$ has rank one for $D
_{\mathcal{\widetilde{H}}_{\sigma}}^Su$-a.e. $x \in \Omega $.
\end{theorem}

\begin{proof}
We adopt the method of main results in \cite {MV}  or \cite{DMV} to
give the proof.  Let $u = ({u_1}, \ldots ,{u_m}) \in {{\mathcal
B}}{{{ \mathcal V}}^R_{\mathcal{\widetilde{H}}_{\sigma}}}(\Omega,
{\mathbb R^m})$.  For any $i = 1, \ldots ,m$, denote by $D
_{\mathcal{\widetilde{H}}_{\sigma}}^S{u_i} = {\sigma _i} | {D _
{\mathcal{\widetilde{H}}_{\sigma}}^S {u_i}}  |$ for a $ |
{D_{\mathcal{\widetilde{H}}_{\sigma}}^S{u_i}}  |$-measurable map
${\sigma _i}:\Omega  \to {\mathbb S^ {2d - 1}}$. Note that the
equality ${\sigma _i} = {\sigma _{{u_i}}}$ holds $ | {D
_{\mathcal{\widetilde{H}}_{\sigma}}^S{u_i}}  |$-almost everywhere
 using the notation in Section 4. Let
 $${S_i}: =  \Big\{ {(x,t) \in
\Omega  \times \mathbb R:\ t < {u_i}(x)}  \Big\}$$ be the subgraph
of ${{u_i}}$. By Theorem \ref{subgraph-1},  ${S_i}$ has finite
${{\widetilde P}_{\mathcal{\widetilde{H}}_{\sigma}}}$-perimeter in
$\Omega \times \mathbb R$. For convenience, we define the
$\mathcal{\widetilde{H}}_{\sigma}$-reduced boundary of ${S_i}$ as
${\partial _{\mathcal{\widetilde{H}}_{\sigma}}}{S_i}$ and write
${v_i} = {v_{{E_i}}}$ for the measure-theoretic inner normal to
${S_i}$. By Theorem \ref{subgraph-2} and Remark \ref {re-1}, we have
\begin{equation*}
     | {D _{\mathcal{\widetilde{H}}_{\sigma}}^S{u_i}}  |= {\pi _\# }
    ({\mathcal H^{d}} \llcorner {E_i}),
\end{equation*}
where ${E_i}: =  \{ {p \in {\partial
_{\mathcal{\widetilde{H}}_{\sigma}}} {S_i}:{{({v_i}(p))}_{2d + 1}} =
0}
 \}$ and ${\pi _\# }$ denotes push-forward of measures defined
in Section 4. The set ${E_i}$ is  rectifiable and we can assume that
it is contained in the union ${ \cup _{h \in \mathbb N}}\Sigma _h^i$
of $C^1$ hypersurfaces $\Sigma _h^i$ in ${\mathbb R^{d + 1}}$.

By Theorem \ref{subgraph-2}, Remark \ref {re-1} and Lemma
\ref {sub-1},  we apply the well-known properties of rectifiable sets to
conclude that the following properties hold for ${\mathcal H^
{d}}$-a.e.{\rm{~}}$p \in {E_1} \cup \ldots  \cup {E_m}$:
\begin{equation}\label{eq-4.1}
    {v_{{\partial _{\mathcal{\widetilde{H}}_{\sigma}}}{S_i}}}(p) =
    ({\sigma _i}(\pi (p)),0);
\end{equation}
\begin{equation}\label{eq-4.2}
    {\rm{if~}}p \in \Sigma _h^i, {\rm{~then~}}{v_i}(p)=
    \pm {v_{\Sigma _h^i}}(p);
\end{equation}
\begin{equation}\label{eq-4.3}
    {\rm{if~}}p \in \Sigma _h^i{\rm{~and~}}q \in {E_j} \cap
    \Sigma _k^j \cap
    {\pi ^{ - 1}}(\pi (p)),{\rm{~then~}}{v_{\Sigma _h^i}}(p)=
    \pm {v_{\Sigma _k^j}}(q).
\end{equation}
Via modifying ${E_i}$ on an ${\mathcal H^{d}}$-negligible set and
${\sigma _i}$ on a $ |{D _{\mathcal{\widetilde{H}}_{\sigma}}^S{u_i}}
|$-negligible set, we can assume that the properties
(\ref{eq-4.1})-(\ref{eq-4.3}) hold everywhere on ${{E_i}}$ and
${\sigma _i}=0$ on $\Omega \backslash \pi ({E_i})$.

Since $D_{\mathcal{\widetilde{H}}_{\sigma}}^Su = ( {{\sigma _1} |
{D_{\mathcal{\widetilde{H}}_{\sigma}}^S{u_1}}  |, \ldots ,{\sigma
_m} | {D_{\mathcal{\widetilde{H}}_{\sigma}}^S{u_m}}  |} )$ and $ |
{D _{\mathcal{\widetilde{H}}_{\sigma}}^S{u}}  |$ is concentrated on
the union $\pi ({E_1}) \cup  \ldots \cup \pi ( {E_m})$, so we just
need to prove that the matrix-valued function $({\sigma _1}, \ldots
,{\sigma _m})$ has rank one on the set $\pi ({E_1}) \cup \ldots \cup
\pi ({E_m})$. The proof of the following fact
\begin{equation*}
i,j \in \{ 1, \ldots ,m\} ,{\rm{~}}i \ne j,{\rm{~}}x \in \pi ({E_i})
\Longrightarrow {\sigma _j}(x) \in \{ 0,{\sigma _i}(x),- {\sigma
_i}(x)\},
\end{equation*} derives  the desired result.

If $i, j, x$ are given as above and $x \notin \pi ({E_j})$, then
${\sigma _j} (x) = 0$. Otherwise, $x \in \pi ({E_i}) \cap \pi
({E_j})$, i.e., there exist $p \in {E_i}$ and $h \in \mathbb N$ such
that $\pi (p) = x$ and ${\sigma _i}(x)= \pm {v_{\Sigma _h^i}}(p)$. Also,
 there exist $q \in {E_j}$ and $k \in \mathbb N$ such that $\pi
(q) = x$ and ${\sigma _j}(x)= \pm {v_{\Sigma _k^j}}(p)$. From
(\ref{eq-4.3}), we conclude that ${\sigma _j}(x)= \pm {\sigma
_i}(x)$. This completes the proof of Theorem \ref{ Rand-one }.
\end{proof}

\section{Acknowledgements}
\hspace{0.4cm}{ J.Z. Huang was supported  by the Fundamental
Research Funds for the Central Universities  (No.\,500421126).}

{P.T. Li was supported by the National Natural Science Foundation of
China   ( No.\,11871293),  Shandong Natural Science Foundation of
China (No.\,ZR2017JL008) and University Science and Technology
Projects of Shandong Province (No.\,J15LI15).}

{Y. Liu was supported by the National Natural Science Foundation of
China (No.\,11671031)  and Beijing Municipal Science and Technology
Project (No.\,Z17111000220000). }

\hspace{-0.65cm}
{\bf Address: }

\flushleft   \flushleft Yang Han\\
          School of Mathematics and Physics\\
          University of Science and Technology Beijing\\
          Beijing 100083,  China\\
          E-mail: hanyang697@163.com

\flushleft Jizheng Huang\\
School of Science \\
Beijing University of Posts and Telecommunications\\
 Beijing 100876,
 China\\
 E-mail address: hjzheng@163.com

          \flushleft Pengtao Li\\
          School of Mathematics and Statistics\\
           Qingdao University\\
            Qingdao,  Shandong 266071,  China\\
E-mail address: ptli@qdu.edu.cn

 \flushleft Yu Liu\\
          School of Mathematics and Physics\\
          University of Science and Technology Beijing\\
          Beijing 100083,  China\\
          E-mail: liuyu75@pku.org.cn

\end{document}